\documentclass[10pt,oneside,a4paper,reqno]{amsart}
\usepackage[a4paper,margin=3cm]{geometry}
\usepackage[font=small]{caption} 
\usepackage{graphicx}
\graphicspath{{./pics}{./}}
\usepackage{comment}
\usepackage{cite}
\usepackage{amssymb,amsfonts,amsbsy}
\usepackage{mathtools}
\usepackage{float}
\usepackage{epsfig}
\usepackage{esint}
\usepackage{cancel}
\usepackage[dvipsnames]{xcolor}

\newcommand{\R}{{\mathbb{R}}}

\newcommand{\LL}{{\mathcal{L}}}

\renewcommand{\d}{{\mathrm d}}

\newcommand{\verti}[1]{\ensuremath{\left\lvert#1\right\rvert}}

\newcommand{\mcl}{m_{\text{\sc cl}}}

\newtheorem{Theorem}{Theorem}[section]

\newtheorem{Remark}[Theorem]{Remark}

\newcommand{\ch}[1]{#1}%{{\color{green!80!black}{#1}}}

\newcommand{\ignore}[1]{}

\newcommand{\st}[1]{}

\renewcommand{\ignore}[1]{}

\numberwithin{equation}{section}
\begin{document}

\title{Droplet motion with contact-line friction:
long-time asymptotics in complete wetting}

\author{Lorenzo Giacomelli}
\address[Lorenzo Giacomelli]{SBAI Department, Sapienza University of Rome - Via A. Scarpa 16, 00161 Roma, Italy - lorenzo.giacomelli@uniroma1.it}

\author{Manuel V.~Gnann}
\address[Manuel V.~Gnann]{Delft Institute of Applied Mathematics, Faculty of Electrical Engineering, Mathematics and Computer Science, Delft University of Technology, Mekelweg 4, 2628 CD Delft, The Netherlands}

\author{Dirk Peschka}
\address[Dirk Peschka]{Weierstrass Institute - Mohrenstrasse 39, 10117 Berlin, Germany - dirk.peschka@wias-berlin.de}

\thanks{LG acknowledges discussions with Maria Chiricotto. MVG appreciates discussions with Jochen Denzler, Robert McCann, and Christian Seis regarding self-similar asymptotics preceding the preperation of this work. MVG was partially supported by the Deutsche Forschungsgemeinschaft (DFG) under project \# 334362478. DP thanks Luca Heltai and Marita Thomas for fruitful discussions and acknowledges the financial support within the DFG-Priority Programme 2171 by project \# 422792530.}

\begin{abstract}
    We consider the thin-film equation for a class of free boundary conditions modelling friction at the contact line, as introduced by E and Ren. Our analysis focuses on formal long-time asymptotics of solutions in the perfect wetting regime. In particular, through the analysis of quasi-self-similar solutions, we characterize the profile and the spreading rate of solutions depending on the strength of friction at the contact line, as well as their (global or local) corrections, which are due to the dynamical nature of the free boundary conditions. These results are complemented with full transient numerical solutions of the free boundary problem.
\end{abstract}

\maketitle

% 35C06 Self-similar solutions to PDEs
% 35K65 Degenerate parabolic equations
% 76M10 Finite element methods applied to problems in fluid mechanics
% 76A20 Thin fluid films

\section{Introduction}

\subsection{\ch{Thin-film equations}}

Thin-film equations are a class of fourth-order degenerate\ch{-}parabolic equations whose prototype, in one space dimension, is
\begin{equation}\label{TFE}
    \partial_t h + \partial_y (m(h) \partial_y^3 h) =0 \quad\mbox{on }\ \{h>0\}:=\{(t,y)\in (0,\infty)\times \R:\ h(t,y)>0\},
\end{equation}
where $m$ is a mobility which degenerates at $h=0$.

\medskip

Thin-film equations are formally derived as leading-order approximations of the Navier-Stokes equations in a suitable regime, which is known as \ch{the} {\em lubrication approximation} \cite{OronDavisBankoff1997,BonnEggersIndekeuMeunierRolley2009,GuentherProkert2008}. In this case, $h$ represents the height of a thin layer or droplet of a Newtonian fluid on a flat solid substrate and the mobility $m$ \ch{often} has the form $m(h)=h^3 +b^{3-n} h^n$, the parameters $b$ and $n$ being related to \ch{a} slip condition imposed at the liquid-solid interface: in particular, $n = 1$ corresponds to Greenspan's slip condition \cite{Greenspan}, $n=2$ corresponds to (linear) Navier slip and $n=3$ (or $b=0$) corresponds to no slip at the substrate. The case $m(h)=h$ may also be seen as the lubrication approximation of the two-dimensional Hele-Shaw flow in the half-space \cite{GO2}, and in this case the lubrication approximation has been given rigorous justifications \cite{GO-var,GO2,KnuepferMasmoudi2013,KM,MatiocProkert2012}.
Nonlinear free boundary problems with $m(h)=1$ are discussed in the context of surface diffusion \cite{dziwnik2014stability}. General traveling wave solutions for $m(h)=h^n$ with $0\le n< 3$ and beyond are discussed in \ch{\cite{BKO,king2001moving,ggo.2016}}.

\medskip

We are interested in situations in which an interface exists which separates a dry region of the substrate from a wetted one:
\begin{equation}\label{support}
    \{h(t,\cdot)>0\}=(s_-(t),s_+(t)).
\end{equation}
In this case, \eqref{TFE} is complemented by two obvious boundary conditions at $s_\pm(t)$, respectively: the defining condition
\begin{equation}\label{BC1}
    h(t,s_\pm(t)) =0,
\end{equation}
and the kinematic condition
\begin{equation}\label{BC2}
    \big(h^{-1} m(h) \partial_y^3 h\big)|_{y= s_\pm(t)} =\dot s_\pm(t),
\end{equation}
where here and after we agree that
$$
%
%f|_{y=a_\pm}= \lim_{y\to a_\pm^\mp}f(y).
%
\ch{
f|_{y=s_+}= \lim_{\substack{y\to s_+ \\ y<s_+}}f(y), \qquad f|_{y=s_-}= \lim_{\substack{y\to s_- \\ y>s_-}}f(y).
}
$$

We are interested in situations in which the support \ch{$\{h(t,\cdot)>0\}$} is allowed to evolve in time, leading to a genuine free boundary problem for which a third condition is required. Appropriate choices of such a third boundary condition are being debated since decades by now. \ch{Starting with the work of Bernis and Friedman \cite{bf.1990}, most of the analytical theory concentrated on the condition of constantly zero contact angle, focusing on existence of weak solutions \cite{BBD,BP,bdgg.1998,g.2004,ag.2004}, on their qualitative properties \cite{b.1996,b.1996.2,hs.1998,dgg.2001,dgs.2001,g.2003,g.2004.aihp,gs.2005,dgg.2006,f.2013,f.2014,f.2016,DNF.2022}, and on well-posedness in weighted spaces \cite{GKO,BGKO2016,ggo.2013,Gnann2015,BelgacemGnannKuehn2016,Gnann2016,seis,ggko,GnannPetrache2018,GnannIbrahimMasmoudi2019}. For the constant, non-zero contact angle case we refer to \cite{o.1998,BGK, M15, Knuepfer2011,Knuepfer2015,Degtyarev2017,Knuepfer2022,Esselborn2016,MajdoubMasmoudiTayachi2021}.
For quasistatic models of droplet evolution we refer to \cite{semprebon2014onset}  for the Stokes flow and to \cite{Greenspan,GrunewaldKim2011,DeMe21} for thin-film models.
}

\smallskip

%Of interest to us is
\ch{We focus on }
a class of contact-line conditions, first considered in \cite{Greenspan,EhrhardDavis} in special cases, %introduced in \cite{RenE2007},
\ch{that} relate the contact-line velocity, $\dot s$, to the dynamic contact angle. \ch{Analytical works for this class are limited to a few contributions \cite{CG-simai,CG-cms,CG-ifb}.} The class has been motivated and generalized by E and Ren building up on a simple, basic principle: consistency with the second law of thermodynamics in the isothermal case \cite{RenE2007,RHE,RenE.2011}. \ch{We} now introduce \ch{this class} directly at the level of lubrication theory\ch{\st{,}.}

%
%%%%%%%%%%%%%%%%%%%%
\subsection{\ch{Contact-line frictional laws in complete wetting}}
%%%%%%%%%%%%%%%%%%%%

At leading order in lubrication approximation, and after a normalization, the surface energy reads
\begin{equation}\label{def-E}
    \mathcal E[h(t)]=  \int_{s_-(t)}^{s_+(t)} \big(\tfrac12(\partial_y h)^2 -S\big) \d y,
\end{equation}
where $S\in \R$ is (proportional to) the so-called spreading coefficient \cite{GO-var}. We are interested in a regime where the spreading coefficient of the solid/liquid/vapor system vanishes, i.e., the complete wetting regime $S=0$. This is a generic \ch{situation} in the so-called ``moist'' case, which concerns for instance a surface which has been pre-exposed to vapor \cite{deGennes}: thus
\begin{equation}\label{def-E0}
    \mathcal E[h(t)]= \tfrac12 \int_{s_-(t)}^{s_+(t)} (\partial_y h)^2 \d y,
\end{equation}
in what follows (for the case $S\ne 0$, see e.g. the discussions in \cite{O,bgk.2005,M15,DurG}).
After integrations by parts (see \cite[formula (1.13)]{CG-ifb}), \eqref{TFE}-\eqref{BC2} formally yield
\begin{align}\nonumber
    \tfrac{\d}{\d t} \mathcal E[h(t)] =& \dot{s}_+(t)\tfrac12 (\partial_y h)^2|_{y=s_+} -\dot{s}_-(t)\tfrac12 (\partial_y h)^2|_{y=s_-}+\int_{s_-(t)}^{s_+(t)}\partial_y h(\partial_{t}\partial_y h)\d y\\
    =& - \dot s_+(t)\big(\tfrac12(\partial_y h)^2 - h \partial_y^2 h\big)|_{y=s_+}
   + \dot s_-(t)\big(\tfrac12(\partial_y h)^2 -  h \partial_y^2 h\big)|_{y=s_-} \nonumber \\
   &- \int_{s_-(t)}^{s_+(t)}\!m(h) (\partial_y^3 h)^2 \, \d y.
    \label{dissHeur}
\end{align}
If by contradiction $ h (\partial_y^2h) |_\ch{y=s_{\pm}(t)} =: k_{\pm} \ne 0$ then one would have
\begin{align*}
\frac12 \partial_y \big((\partial_y h)^2\big) &= (\partial_y h) (\partial_y^2 h)\sim k_{\pm} h^{-1} \partial_y h = k_{\pm} \partial_y \log h && \mbox{as} \quad y\to s_{\pm}(t)^{\mp},
\end{align*}
whence $(\partial_y h)^2$ would become unbounded as $y\to s_{\pm}(t)^{\mp}$. Hence, $h (\partial_y^2 h) |_\ch{y = s_{\pm}(t)} = 0$ for solutions with finite slope at the contact line,
which is the class we are interested in. Therefore, \eqref{dissHeur} reads as
\begin{equation}\label{dissHeur2}
\tfrac{\d}{\d t} \mathcal E[h(t)] = -\tfrac12 \dot s_+(t)(\partial_y h)^2|_{y=s_+(t)} + \tfrac12\dot s_-(t)(\partial_y h)^2|_{y=s_-(t)}  - \int_{s_-(t)}^{s_+(t)} m(h) (\partial_y^3 h)^2 \, \d y.
\end{equation}
Consistency with the second law of thermodynamics implies that the first two terms on the right-hand side of \eqref{dissHeur2} must be non-positive. A simple form of constitutive relations which enforces non-positivity is
\begin{equation}\label{bc-pre}
(\partial_y h)^2|_\ch{x= s_\pm(t)} = \ch{ f_{\ell} (\pm \dot{s}_\pm)} \quad\mbox{with}\ f_\ell(\sigma)\sigma \ch{\ge 0} \quad \mbox{for all $\sigma\in \R$.}
%
%\mp f_{\ell} (\dot{s}_\pm) \quad\mbox{with \ } f_\ell(\sigma)\sigma \le 0 \mbox{ \ for all $\sigma\in \R$.}
%
\end{equation}
\ch{According to \eqref{bc-pre}, receding fronts with speed $\sigma<0$ only exist if $f_\ell(\sigma)=0$. Hence, in complete wetting, receding fronts (if any)}
%
%In particular, \eqref{bc-pre} requires that $f_\ell(\sigma)=0$ for $\sigma\le 0$, i.e., receding droplets in the complete wetting regime
%
have zero contact angle.
Furthermore, note that the more standard zero contact-angle condition $(\partial_y h)(t,s_\pm(t))=0$ corresponds to the \ch{limit} of \ch{vanishing} contact-line dissipation, $f_\ell \equiv 0$. For the analogue of \eqref{bc-pre} in partial wetting ($S<0$) we refer to \cite{CG-ifb,peschka2018variational}.

\smallskip

The simple argument above can in fact be embedded into a more general formal gradient-flow structure of the system, based on a separation of dual forces into a bulk and a contact-line dissipation. This structure, which we elaborate in \ch{Appendix}~\ref{sec:numerics}, is also at the basis of the discretization that we adopt in numerical simulations.

%%%%%%%%%%%%%%%%%%%%%%%%%%%
\subsection{\ch{Goals}}\label{ss:goals}
%%%%%%%%%%%%%%%%%%%%%%%%%%%

Intermediate asymptotics for \eqref{TFE} with the Ren-E boundary condition \eqref{bc-pre}, such as the \ch{Voinov-Cox-Hocking law} (see the discussions in \cite{BDDG,GO,AG1,ggo.2016,GnannWisse2022}), has been formally worked out in \cite{CG-cms}. Instead, here we will focus on the long-time dynamics. In complete wetting, it is expected that generic solutions spread indefinitely, covering the whole real line with a layer of zero-thickness in the limit as $t\to +\infty$. Therefore, the long-time dynamics can be equivalently captured by considering a power-law form $m(h) = h^n$ of the mobility, instead of its full form $m(h)=h^3+b^{3-n}h^n$. Then \eqref{TFE} reads
\begin{equation}\label{TFE-n}
    \partial_t h + \partial_y (h^n \partial_y^3 h) =0 \quad\mbox{on}\ \{h>0\}, \qquad n\in [1,3)\ch{\mbox{\st{.},}}
\end{equation}
and we are led to consider the following free boundary problem:
\begin{subequations}\label{tfe_classical_unscaled}
    \begin{align}
        \partial_t h + \partial_y \big(h^n \partial_y^3 h\big) & = 0                                      &  & \mbox{for} \quad y \in (s_-(t),s_+(t)), \label{tfe_classical_1_unscaled} \\
        h                                                   & = 0                                      &  & \mbox{at} \quad y = s_\pm(t), \label{tfe_classical_2_unscaled}           \\
        (\partial_y h)^2                                    & \ch{= f_\ell(\pm \dot s_\pm(t))} &  & \mbox{at} \quad y = s_\pm(t), \label{tfe_classical_3_unscaled}           \\
%
%        (\partial_y h)^2                                    & =d \left(\pm \dot s_\pm(t)\right)^\alpha &  & \mbox{at} \quad y = s_\pm(t), \label{tfe_classical_3_unscaled}           \\
%
        h^{n-1} (\partial_y^3 h)                            & = \dot s_\pm(t)                          &  & \mbox{at} \quad y = s_\pm(t)\ch{.} \label{tfe_classical_4_unscaled}
    \end{align}
\end{subequations}
In the absence of contact-line friction\ch{\st{, the thin-film equation} ($f_\ell\equiv 0$), \eqref{tfe_classical_unscaled}} admits self-similar solutions \cite{BPW}, which are expected to describe the long-time dynamics of generic solutions (however, rigorous results are available for $n=1$ only \cite{CarrilloToscani2002,carlen-ulusoy,MMCS,seis,CarlenUlusoy2014,Gnann2015}).
\ch{\st{In general, {\eqref{bc-pre}}} If $f_\ell\not\equiv 0$, \eqref{tfe_classical_3_unscaled}} breaks the self-similar structure of \ch{\st{{\eqref{TFE-n}}}\eqref{tfe_classical_1_unscaled}}. Nevertheless, long-time dynamics may be inferred from the analysis of (quasi-)self-similar solutions, where the non-self-similar part of the operator is seen as a small modulation in time. This method has already been applied to thin-film equations for related asymptotic studies \cite{AG1,BDDG}.

%%%%%%%%%%%%%%%%%%%%
\subsection{\ch{The model problem}}
%%%%%%%%%%%%%%%%%%%%

We assume \ch{\st{a}} prototypical power-law form\ch{s} of $f_\ell$, i.e.,
\ch{\begin{equation}\label{BC2-bis}
    f_\ell(\sigma)=dg_\alpha(\sigma),\qquad \mbox{where} \quad g_\alpha(\sigma)=\max\big\{0,|\sigma|^{\alpha-1}\sigma\big\} \quad\mbox{or}\quad  g_\alpha(\sigma) = |\sigma|^{\alpha-1}\sigma,
\end{equation}
}
\ch{and} $\alpha > 0$ and $d>0$ are constants encoding the strength of friction at the contact line. \ch{For $S=0$,  the former $g_\alpha$ allows for receding fronts, while the latter alternative does not.} \ch{\st{Therefore, {\eqref{bc-pre}} assumes the form}
%\begin{equation*}%\label{BC2-bis}
%    \cancel{\big(\partial_y h|_{y= s_\pm(t)}\big)^2 = %\ch{d \, g_\alpha(\pm \dot s_\pm(t))}
%    d\left(\pm \dot s_\pm(t)\right)^\alpha}
%\end{equation*}
}
\begin{comment}
and we are led to consider the following free boundary problem:
\begin{subequations}\label{tfe_classical_unscaled}
    \begin{align}
        \partial_t h + \partial_y \big(h^n \partial_y^3 h\big) & = 0                                      &  & \mbox{for} \quad y \in (s_-(t),s_+(t)), \label{tfe_classical_1_unscaled} \\
        h                                                   & = 0                                      &  & \mbox{at} \quad y = s_\pm(t), \label{tfe_classical_2_unscaled}           \\
        (\partial_y h)^2                                    & \ch{= d \, g_\alpha(\pm \dot s_\pm(t))} &  & \mbox{at} \quad y = s_\pm(t), \label{tfe_classical_3_unscaled}           \\
%
%        (\partial_y h)^2                                    & =d \left(\pm \dot s_\pm(t)\right)^\alpha &  & \mbox{at} \quad y = s_\pm(t), \label{tfe_classical_3_unscaled}           \\
%
        h^{n-1} (\partial_y^3 h)                            & = \dot s_\pm(t)                          &  & \mbox{at} \quad y = s_\pm(t), \label{tfe_classical_4_unscaled}
    \end{align}
\end{subequations}
\end{comment}

Equation \eqref{tfe_classical_1_unscaled} and the kinematic condition \eqref{tfe_classical_4_unscaled} have two scaling invariances,
\begin{equation}\label{scaling1}
    (t,y,h)\mapsto (T_*\hat t, Y_*\hat y, H_* \hat h), \quad \mbox{where} \quad T_* = Y_*^{4} H_*^{-n}.
\end{equation}
We use one \ch{invariance} to normalize the droplet's mass to be $2$, \ch{i.e.}
\begin{equation}\label{scaling2}
    M :=\int_{s_-(t)}^{s_+(t)} \ch{h(t,y)}\,\d y = 2 H_* Y_*.
\end{equation}

With \ch{\st{this choice, {\eqref{BC2-bis}} reads as}} \ch{the choice in \eqref{BC2-bis}, \eqref{tfe_classical_3_unscaled} reads as}
\[
    \big(\partial_{\hat y} \hat h|_{\hat y= \hat s_\pm(\hat t)}\big)^2 =
    %\underbrace{d\frac{Y_*^2}{H_*^2}\frac{Y_*^\alpha}{T_*^\alpha}}_{=:D^2}
    D^2 \ch{g_\alpha\big(\pm \tfrac{\d \hat s_\pm}{\d \hat t}(\hat t)\big)} \ch{\qquad\text{where}\qquad D^2:=d\frac{Y_*^2}{H_*^2}\frac{Y_*^\alpha}{T_*^\alpha}}.
    %
    %\big(\pm \tfrac{\d \hat s_\pm}{\d \hat t}(\hat t)\big)^\alpha.
    %
\]
If $\alpha\ne \frac{4}{n+3}$, we may use the second scaling invariance to fix the constant $D>0$ to be $1$:
\begin{equation}\label{D=1}
    1=D^2= \frac{dY_*^{2+\alpha}}{H_*^2 T_*^\alpha} \stackrel{\eqref{scaling1}}=\frac{dY_*^{2-3\alpha}}{H_*^{2-\alpha n}} \stackrel{\eqref{scaling2}}= \frac{d\big(\tfrac{M}{2}\big)^{2-3\alpha}}{H_*^{4-\alpha (n+3)}} \quad \Leftrightarrow \quad H_*^{4-\alpha (n+3)}= d\big(\tfrac{M}{2}\big)^{2-3\alpha}.
\end{equation}
Therefore, removing hats, \ch{\st{in what follows}} we will consider the following free boundary problem with initial mass $M=2$ and with $D=1$ if $\alpha\ne \frac{4}{n+3}$:
\begin{subequations}\label{tfe_classical}
    \begin{align}
        \partial_t h + \partial_y \big(h^n \partial_y^3 h\big) & = 0                                    &  & \mbox{for} \quad y \in (s_-(t),s_+(t)), \label{tfe_classical_1} \\
        h                                             & = 0                                    &  & \mbox{at} \quad y = s_\pm(t), \label{tfe_classical_2}           \\
     \big(\tfrac{1}{D} \partial_y h\big)^2      & =\ch{g_\alpha(\pm \dot s_\pm(t))} &  & \mbox{at} \quad y = s_\pm(t), \label{tfe_classical_3}           \\
%
%        \big(\tfrac{1}{D} \partial_y h\big)^2      & =\left(\pm \dot s_\pm(t)\right)^\alpha &  & \mbox{at} \quad y = s_\pm(t), \label{tfe_classical_3}           \\
%
        h^{n-1} \partial_y^3 h                        & = \dot s_\pm(t)                        &  & \mbox{at} \quad y = s_\pm(t), \label{tfe_classical_4} \\
       \int_{s_-(t)}^{s_+(t)} h(\ch{t},y)\d y  & =2.                                           &  & \label{tfe_classical_mass}
    \end{align}
\end{subequations}
We seek {\ch{\st{for}} even solutions, with $s=s(t)=s_+(t) = - s_-(t)$ denoting the position of the right free boundary. Furthermore, \ch{since we are interested in the long-time dynamics, we seek solutions with advancing contact lines $\dot{s}>0$, whence \eqref{BC2-bis} reduces to $g_\alpha(\dot s)=(\dot s)^\alpha$. It} is convenient to pass to a fixed domain by the change of variables
\begin{equation}\label{scaling_ansatz}
h(t,y)=s^{-1} H(t,x), \quad x=s^{-1}y.
\end{equation}
Then, taking symmetry into account, \eqref{tfe_classical} reads as
\begin{subequations}\label{tfe_classical_fixed}
    \begin{align}
        s^{n+4} \partial_t H -s^{n+3}\dot s \partial_x (x H)+ \partial_x \big(H^n \partial_x^3 H\big) & = 0                                    &  & \mbox{for} \quad x \in (0,1), \label{tfe_classical_1_fixed} \\
        H                                             & = 0                                    &  & \mbox{at} \quad x = 1, \label{tfe_classical_2_fixed}           \\
%
%         s^{-4}\big(D^{-1} \partial_x H\big)^2      & =\ch{g_\alpha(\dot s)} &  & \mbox{at} \quad x = 1, \label{tfe_classical_3_fixed}           \\
%
        s^{-4}\big(D^{-1} \partial_x H\big)^2      & =\left(\dot s\right)^\alpha &  & \mbox{at} \quad x = 1, \label{tfe_classical_3_fixed}           \\
        s^{-(n+3)} H^{n-1}\partial_x^3 H                        & = \dot s                        &  & \mbox{at} \quad x = 1, \label{tfe_classical_4_fixed}
        \\
        \partial_x H = \partial_x^3 H                        & = 0                        &  & \mbox{at} \quad x = 0, \label{tfe_classical_5_fixed}
        \\
        \int_0^{1} H(\ch{t},x)\d x                        & = 1\ch{\mbox{\st{.},}}                        &  &  \label{tfe_classical_mass_fixed}
    \end{align}
\end{subequations}
\ch{which we will consider in what follows.} Note that the combination of  \eqref{tfe_classical_3_fixed} and \eqref{tfe_classical_4_fixed} yields
\begin{align}
\label{tfe_classical_34_fixed}
    \big(D^{-1} \partial_x H\big)^2 &= s^{4-\alpha(n+3)} (H^{n-1}\partial_x^3 H)^\alpha && \mbox{at} \quad x= 1.
\end{align}
A generic symmetric solution of \ch{\st{the transient}} problem \eqref{tfe_classical} is shown in Fig.~\ref{fig:generic_solution}\ch{, where we choose initial conditions highlighting transient behavior and convergence to self-similar solutions for $h(t,y)$.}
\begin{figure}[b]
    \centering
    \includegraphics[width=\textwidth]{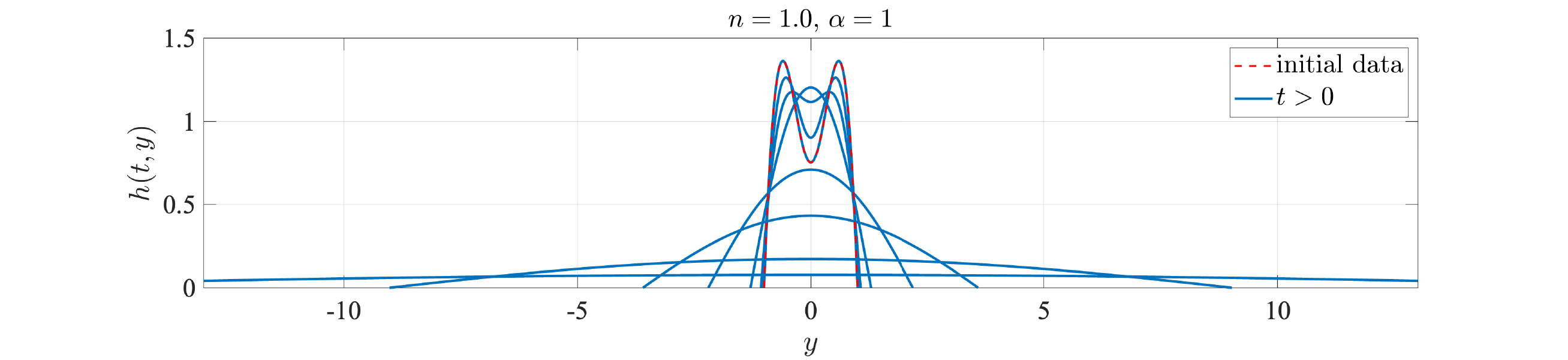}
    \caption{Solution $h(t,y)$ of the transient problem \eqref{tfe_classical} with $n=\alpha=1$ and $D=1$ with (symmetric) initial data $h(t=0,y)=\tfrac{15}{4}((1+y)(1-y)-\tfrac25(1+\cos(\pi y))$ for $s_\pm(t=0)=\pm 1$.}
\label{fig:generic_solution}
\end{figure}
%

%%%%%%%%%%%%%%%%%%%%
\subsection{\ch{Outline}}
%%%%%%%%%%%%%%%%%%%%

It is apparent from \eqref{tfe_classical_34_fixed} that \eqref{tfe_classical_fixed} has a self-similar structure if and only if
\begin{equation}\label{alpha-critical}
 \alpha = \frac{4}{n+3}.
\end{equation}
\ch{
In \S \ref{s:ss} we discuss the case \eqref{alpha-critical}, identifying for any $D\ge 0$ a unique self-similar profile (cf. Theorem \ref{thm}); we also discuss its behavior with respect to $D$ and its stability.
%In \S \ref{s:app_2} we investigate quasi-self-similar solutions and show that the long-time dynamics is dominated by different balances in the dissipation: For $\alpha <\frac{4}{n+3}$ contact-line dissipation dominates and the long-time scaling law depends on $\alpha$. We consider this case of strong contact-line friction in \S \ref{qss:strong}. For $\alpha >\frac{4}{n+3}$ bulk dissipation dominates leading to the same long-time scaling law (depending on $n$) as in the case of null contact-line dissipation. This case of weak contact-line friction is considered in \S \ref{qss:weak}. For \eqref{alpha-critical}, both dissipations contribute equally and we speak of the balanced case.
%
% A summary of the results obtained is given in \S \ref{s:concl}, which also contains their discussion and indicates further directions. Finally, in Appendix \ref{sec:numerics} we detail the gradient-flow formulation of the problem and we discuss how this formulation drives our numerical scheme.
%
In \S \ref{s:app_2} we investigate quasi-self-similar solutions. Our analysis shows that the long-time dynamics is dominated by:

\smallskip

$\bullet$  contact-line dissipation if $\alpha <\frac{4}{n+3}$ (strong contact-line friction, \S \ref{qss:strong}): in particular, the long-time scaling law depends on $\alpha$;

$\bullet$ bulk dissipation if $\alpha >\frac{4}{n+3}$ (weak contact-line friction, \S \ref{qss:weak}), leading to the same long-time scaling law as in the case of null contact-line dissipation.

\smallskip

For \eqref{alpha-critical}, both dissipations contribute equally and we speak of the balanced case. A summary of the results obtained is given in \S \ref{s:concl}, which also contains their discussion and indicates further directions. Finally, in Appendix \ref{sec:numerics} we detail the gradient-flow formulation of the problem and we discuss how this formulation drives our numerical scheme.
}

%%%%%%%%%%%%%%%%%%%%
\section{Self-similar solutions}\label{s:ss}
%%%%%%%%%%%%%%%%%%%%

We seek \ch{\st{for}} symmetric and mass-preserving self-similar solutions of \eqref{tfe_classical} in \ch{the} case \ch{when} the balance condition \eqref{alpha-critical} holds. In terms of $H$, this translates into the ansatz
\begin{equation}\label{ansatz_self}
H(t,x)= H(x)\quad\mbox{and} \quad s(t)= (\gamma^{-1} B^2 t)^{\gamma}, \qquad \gamma=\frac{1}{n+4},
\end{equation}
where $B>0$ is an unknown constant. Inserting \eqref{ansatz_self} into \eqref{tfe_classical_fixed}, integrating once using \eqref{tfe_classical_5_fixed} and recalling \eqref{tfe_classical_34_fixed}, we find
\begin{subequations}\label{ss-new}
    \begin{align}
        H^{n-1} \tfrac{\d^3 H}{\d x^3}       & = B^2 x && \mbox{in} \quad (0,1), \label{ss-new-1}                                                           \\
        \tfrac{\d H}{\d x}                        & = 0 && \mbox{at} \quad x = 0, \label{ss-new-2}                                                               \\
         H                                              & = 0 && \mbox{at} \quad x = 1, \label{ss-new-3}                                                               \\
 \left(\tfrac{\d H}{\d x}\right)^2 & = D^2 (H^{n-1} \tfrac{\d^3 H}{\d x^3})^\alpha \stackrel{(\ref{ss-new-1})}= D^2 B^{2\alpha} && \mbox{at} \quad x= 1, \label{ss-new-4} \\
%        \left(D^{-1} \tfrac{\d \ch{\hat H}}{\d x}\right)^2 & = \Big(\int_0^{1} \ch{\hat H} \ \d x\Big)^{\frac{2(3-n)}{n+3}} && \mbox{at} \quad x = 1. \label{ss-new-4}
\int_0^{1} H(x)\d x                         & = 1,                       &  &  %\label{tfe_classical_mass_fixed}
    \end{align}
\end{subequations}
where we recall that in this case $D$ cannot be set to $1$ by scaling. For $n=1$, we have $\alpha \stackrel{\eqref{alpha-critical}}{=} 1$ and \eqref{ss-new} can be integrated explicitly to the unique solution
\begin{equation}\label{self_exact}
     H_D(x) = \tfrac{B_D D}{2} (1-x^2) + \tfrac{B_D^2}{24} (1-x^2)^2 , \quad B = B_D:=\tfrac{15}{2}\Big(-D +\sqrt{\tfrac45+D^2}\Big).
\end{equation}
In particular,
\begin{equation}\label{conj1}
B_D^2
    \left\{\begin{array}{ll}
        = 45 & \mbox{for $D=0$,}
        \\[1ex]
        \sim 9/D^2       & \mbox{as $D\to \infty$},
    \end{array}\right.
    \quad\mbox{hence}\quad H_D(x) \left\{\begin{array}{ll}
= \frac{15}{8}(1-x^2)^2 & \mbox{for $D=0$}
\\[1ex]
\ch{=} \frac{3}{2}(1-x^2) & \mbox{as $D\to \infty$}.
\end{array}\right.
\end{equation}

\begin{Remark}\label{rem:SH}{\rm We remark that \ch{$H_D$ for $D=0$} coincides with the \emph{Smyth-Hill solution} \cite{SmythHill1988}.
       % Note that $h_{\mathrm S}$ coincides with the zero contact angle profile $\frac{1}{120} t^{-\frac 1 5} (x_0^2-x^2)^2$ for $x \in (-x_0,x_0)$, known as the \emph{Smyth-Hill solution} \cite{SmythHill1988}, in the limit $D\to 0$.
        Further note that the profile \eqref{self_exact}
        was formulated already in \cite[Eq.~(6.1)]{CarrilloToscani2002} and later on in \cite{BelgacemGnannKuehn2016}, without a justification of the contact-angle dynamics, that is, instead of \eqref{ss-new-4} the condition $\tfrac{\d H}{\d x} =$constant at $x=1$ was assumed (cf.~\cite[(2.3)]{BelgacemGnannKuehn2016}).
    }\end{Remark}

For $n>1$, the solution of \eqref{ss-new} is not explicit. However, the situation is analogous:

\begin{Theorem}\label{thm} Let $n\in[1,3)$. For any $D\ge 0$ there exists a unique solution $(B_D,H_D)$ to \eqref{ss-new}. In addition, $B_{\ch{D=0}}>0$, $(B_{\ch{D=0}},H_{\ch{D=0}})$ is the unique solution to \eqref{ss-new} with $D=0$, and
$$
D B_D^\alpha \to - 3 \quad\mbox{and} \quad  H_D(x) \to \tfrac32 (1-x^2) \quad \mbox{ in} \ C([0,1]) \quad \mbox{as} \ D\to +\infty.
$$
\end{Theorem}

\begin{Remark}{\rm
Theorem \ref{thm} shows that for large, \ch{respectively} small, contact-line frictional coefficients the evolution is controlled by the contact-line frictional law, \ch{respectively} the complete wetting regime. In addition, since $s(t)\sim (\gamma^{-1} B_D^2 t)^{\gamma}$ and $B_D\to 0$ as $D\to +\infty$, as $D\to +\infty$ solutions approach a quasi-stationary interface shape with respect to the time-scale $t^\gamma$. The behavior of \ch{$H_D$} for varying $D$ is shown in Fig.~\ref{fig:exact_n1_n2}.
}\end{Remark}

\begin{figure}[h]
\centering
\includegraphics[width=\textwidth]{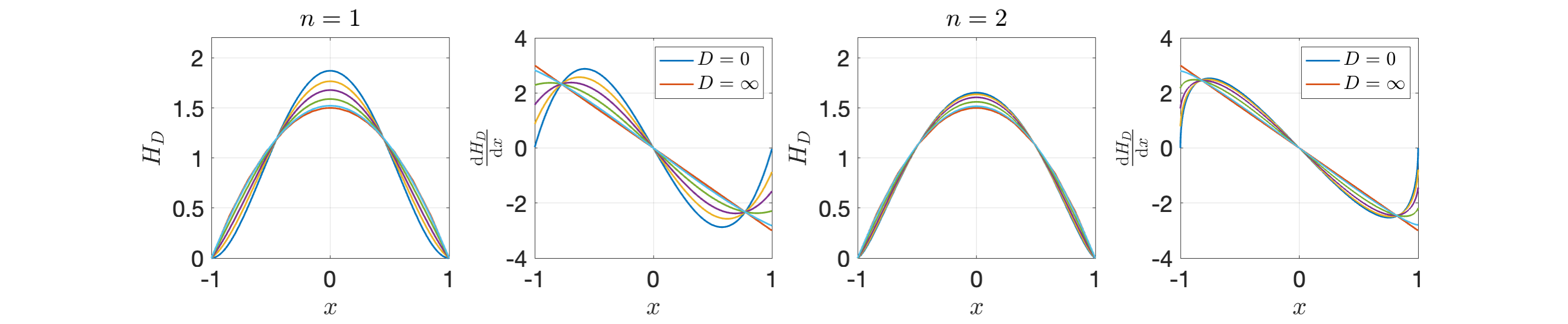}
\caption{Exact self-similar solutions $H_D(x)\equiv H(x)$ \ch{of the ODE system \eqref{ss-new}} and first derivative
%$\partial_x H(x)$
\ch{$\frac{\d H_D}{\d x}$} for $n=1$ (first two panels) and $n=2$ (last two panels) with $\alpha=4/(n+3)$ shown for various friction coefficients $0\le D\le\infty$ ($D = 0$ with $\frac{\d H_D}{\d x} = 0$ at $x = \pm 1$; $D \to + \infty$ with linear $\frac{\d H_D}{\d x}$). \ch{Numerical solutions are obtained by shooting method.}}
\label{fig:exact_n1_n2}
\end{figure}

We conclude the \ch{s}ection with the proof of Theorem \ref{thm}.

\begin{proof}[Proof of Theorem \ref{thm}]
We rescale \eqref{ss-new} as follows:
\begin{equation}\label{back}
\hat x = \hat x_B x \quad \mbox{and} \quad \hat H = \hat x_B^{-1} H, \quad \hat x_B := B^{\frac{2}{n+4}},
\end{equation}
so that $(B,H)$ solves \eqref{ss-new} if and only if $(B,\hat H)$ solves
\begin{subequations}\label{ss-new-scaled}
\begin{align}
\hat H^{n-1} \tfrac{\d^3 \hat H}{\d \hat x^3} &= \hat x && \mbox{for} \quad \hat x \in (0,\hat x_B), \\
\tfrac{\d \hat H}{\d \hat x} &= 0 && \mbox{at} \quad \hat x = 0, \\
(\hat H,\tfrac{\d \hat H}{\d \hat x}) &= (0,-DB^{\alpha-\frac{4}{n+4}})\stackrel{\eqref{alpha-critical}}= (0,-DB^\frac{4}{(n+3)(n+4)} ) && \mbox{at} \quad \hat x = \hat x_B, \\
\int_0^{\hat x_B} u \, \d \hat x &= 1.
\end{align}
\end{subequations}
In \cite[Theorem 2.1]{BDDG}, and in \cite[Theorem~1.2]{BPW} for $D=0$ (up to a rescaling), it is shown that the system
\begin{subequations}\label{sys_bddg}
\begin{align}
u^{n-1} \tfrac{\d^3 u}{\d y^3} &= y && \mbox{for} \quad y \in (0,y_\theta), \\
\tfrac{\d u}{\d y} &= 0 && \mbox{at} \quad y = 0, \\
(u,\tfrac{\d u}{\d y}) &= (0,-\theta) && \mbox{at} \quad y = y_\theta, \label{sys_bddg_34}\\
\int_0^{y_\theta} u \, \d y &= 1.
\end{align}
\end{subequations}
has for any $\theta \ge 0$ a unique solution $(y,u) = (y_\theta, u_\theta)$
(see~\cite[Theorem~3.2 and Theorem~3.3]{BelgacemGnannKuehn2016}
for an alternative proof of existence and uniqueness). Furthermore, $y_\theta$ is a decreasing function of $\theta$ with $0 < y_0 < \infty$ and $y_\infty = 0$. Note that the proof in \cite{BDDG} applies to inhomogeneous mobilities, but carries over to our situation without any change of the reasoning.

If $D=0$, \eqref{ss-new-scaled} coincides with \eqref{sys_bddg} with $\theta=0$, hence Theorem~\ref{thm} holds with $(\hat x_B,\hat H)=(y_0,u)$. If $D>0$, for any $B>0$ let
$$
\theta_B := DB^\frac{4}{(n+3)(n+4)}
$$
and let $(y_B, u_B)$ be the unique solution to \eqref{sys_bddg} with $\theta=\theta_B$. Since $B\mapsto \theta_B$ is increasing, in view of the above, $B\mapsto y_B$ is decreasing from $y_0$ to $y_\infty = 0$: hence there exists a unique $B$ (whence a unique $\theta$) such that $y_B=\hat x_B=B^\frac{2}{n+4}$. This proves the well-posedness of \eqref{ss-new}.

It remains to consider the limit as $D\to +\infty$. Let $\theta_D$, $B_D$, and $y_D=\hat x_D$ be the unique constants identified above. By construction, they are related by
$$
\theta_D= D B_D^\frac{4}{(n+3)(n+4)}\quad\mbox{and}\quad  y_D=B_D^\frac{2}{n+4}.
$$
If by contradiction $\theta_D\to \overline \theta\in [0,+\infty)$ for a subsequence $D\to +\infty$ (not relabeled), then on one hand $B_D\to 0$, hence $y_D\to 0$; on the other hand, \eqref{sys_bddg} would imply that $y_D\to \overline y\in (0,y_0]$, a contradiction. Therefore $\theta_D\to +\infty$, hence $y_D=B_D^{\frac{2}{n+4}}\to 0$ as $D\to +\infty$. %In the limit $B\to 0$
Hence, in the limit $D \to + \infty$, it follows from  \eqref{ss-new} because of $B_D \to 0$ that \ch{$H_D(x)$} converges uniformly in $[0,1]$ to the unique solution of
$$
\tfrac{\d^{3} H}{\d x^{3}} =0 \ \mbox{ in $(0,1)$}, \quad H(1)=\tfrac{\d H}{\d x}(0) =0,\quad \int_0^1 H(x)\d x =\ch{1}.
$$
Hence, $DB_D^{\alpha}\to -3$ and $\ch{H_D(x)}\to \tfrac32(1-x^2)$ \ch{as $D\to\infty$}.
\end{proof}

\section{Quasi-self-similar solutions}\label{s:app_2}
\subsection{Scaling and ansatz}
As we have just seen, only the case $\alpha=\frac{4}{n+3}$ yields exact self-similar solutions to \eqref{tfe_classical_fixed}. For $\alpha\ne \frac{4}{n+3}$, we may nevertheless \ch{consider} quasi-self-similar solutions, which are expected to describe the large-time asymptotics of generic ones. \ch{As mentioned in \S \ref{ss:goals},  this method has been applied to thin-film equations for related asymptotic studies \cite{AG1,BDDG}}. Quasi-self-similar solutions are characterized by ignoring the explicit dependence of $H$ on time in \eqref{tfe_classical_fixed}, that is, the term $\partial_t H$ is dropped while the dependence on time through $s$ is retained. We may then integrate \eqref{tfe_classical_1_fixed} using the boundary conditions \eqref{tfe_classical_2_fixed}, \eqref{tfe_classical_4_fixed}, and $\verti{\dot s} < \infty$.
%Furthermore, we may combine the boundary conditions \eqref{tfe_classical_3_fixed} and \eqref{tfe_classical_4_fixed}
Recalling also \eqref{tfe_classical_34_fixed}, this yields
\begin{subequations}\label{app_tfe}
    \begin{align}
     H^{n-1} \partial_x^3 H & =  s^{n+3} \dot s x                                &  & \mbox{for} \quad x \in (0,1), \label{app_tfe_1} \\
        H                                             & = 0                                    &  & \mbox{at} \quad x = 1, \label{app_tfe_2}           \\
         (\partial_x H)^2  & =s^{4-\alpha(n+3)} (H^{n-1}\partial_x^3 H)^\alpha &  & \mbox{at} \quad x = 1, \label{app_tfe_34}           \\
        \partial_x H & = 0                        &  & \mbox{at} \quad x = 0, \label{app_tfe_5}
        \\
\int_0^1 H\d x & =1, & & \label{app_tfe_mass}
    \end{align}
\end{subequations}
where in view of \eqref{D=1} we have assumed without loss of generality $D=1$.
The free boundary condition \eqref{app_tfe_34} suggests that for $t\gg 1$, i.e. $s\gg 1$, the solution is

\smallskip

$\bullet$ dominated by contact-line friction if $\alpha<\frac{4}{n+3}$ (strong contact-line friction);

$\bullet$  a perturbation of the complete wetting solution if $\alpha>\frac{4}{n+3}$ (weak contact-line friction).

\smallskip

In the next two sections we will show that this is indeed the case. To this aim, we will use the asymptotic expansion
\begin{equation}\label{ansatz_self_n1_strong}
    s^{n+4}(t) = s_0^{n+4}+s_1^{n+4}(1+o(1)) , \qquad H(t,x)= H_0(x) + \omega H_1(x) +O(\omega^2),
\end{equation}
with $s_0(t)\gg 1$ and $s_1(t)\ll s_0(t)$ for $t\gg 1$, and $\omega=\omega(s)\ll 1$ as $s\to \infty$. Under the expansion \eqref{ansatz_self_n1_strong}, we obviously have from \eqref{app_tfe_2} and \eqref{app_tfe_5} that
\begin{equation}\label{bc_h0h1_n1}
        H_0(1) = 0, \qquad \tfrac{\d H_0}{\d x}(0)=0, \qquad
        H_1(1) = 0, \qquad \tfrac{\d H_1}{\d x}(0)=0,
\end{equation}
and from \eqref{app_tfe_mass} that
\begin{equation}\label{mass_n1}
        \int_0^1 H_0 \ \d x  = 1, \qquad
        \int_0^1 H_1 \ \d x  = 0.
\end{equation}
On the other hand, the leading order term in the contact-line condition \eqref{app_tfe_34} depends on the sign of $\alpha-\frac{4}{n+3}$. This motivates distinguishing the two cases.

%%%%%%%%%%%%%%%%%%%%
\subsection{The case $\alpha < \frac{4}{n+3}$: strong contact-line friction}\label{qss:strong}
%%%%%%%%%%%%%%%%%%%%

Since $\alpha<\frac{4}{n+3}$ and $s_0\gg 1$, using \eqref{ansatz_self_n1_strong}, at leading order the bulk equation \eqref{app_tfe_1} and the contact-line condition \eqref{app_tfe_34} \ch{are}
\begin{equation}\label{re_lead_gen_n_strong}
H_0^{n-1} \tfrac{\d^3 H_0}{\d x^3} =  s_0^{n+3}\dot s_0 x \quad \mbox{for} \ x \in (0,1) \qquad\mbox{and}\qquad   H_0^{n-1} \tfrac{\d^3 H_0}{\d x^3} = 0 \quad \mbox{at} \ x = 1,
\end{equation}
respectively. The two equations are obviously incompatible, pointing to the necessity of considering the correction $\omega H_1$. Therefore,
\begin{align}\label{bulk_lead_gen_n_strong}
    \tfrac{\d^3 H_0}{\d x^3} &= 0 && \mbox{in} \quad (0,1)
\end{align}
and the leading-order terms in \eqref{app_tfe_1} and \eqref{app_tfe_34} are, respectively,
\begin{align}
&\omega \big(H_0^{n-1}\tfrac{\d^3 H_1}{\d x^3} + (n-1) H_0^{n-2} \tfrac{\d^3 H_0}{\d x^3} H_1\big) =  s_0^{n+3}\dot s_0 x && \mbox{for} \quad x \in (0,1), \label{re_next_n1} \\
&\big(\tfrac{\d H_0}{\d x}\big)^2 = s_0^{4 - \alpha (n+3)} \omega^\alpha \big(H_0^{n-1}\tfrac{\d^3 H_1}{\d x^3} + (n-1) H_0^{n-2} \tfrac{\d^3 H_0}{\d x^3} H_1\big)^\alpha && \mbox{at} \quad x=1. \label{re_next_gen_n_strong}
\end{align}
Combining \eqref{bc_h0h1_n1}, \eqref{mass_n1}, and \eqref{bulk_lead_gen_n_strong} yields $n$-independently the unique solution
\begin{align}\label{h0_gen_n_strong}
    H_0 &= \tfrac 3 2 (1-x^2) && \mbox{for} \quad x \in [-1,1].
\end{align}
On the other hand, separation of variables in \eqref{re_next_n1} yields with \eqref{bulk_lead_gen_n_strong} and \eqref{h0_gen_n_strong},
\begin{align}
\tfrac{\d^3 H_1}{\d x^3} &= x \, (1-x^2)^{1-n} && \mbox{in} \quad (0,1), \label{h1_gen_n_strong} \\
\omega &= \big(\tfrac23\big)^{n-1} s_0^{n+3}\dot s_0, \label{ome_gen_n}
\end{align}
where, since the expansion \eqref{ansatz_self_n1_strong} only depends on the product $\omega H_1$, we were free to choose the normalization factor of $\omega$. Now, we use
\begin{eqnarray*}
0 &\stackrel{\eqref{mass_n1}}{=}& \int_0^1 \big(\tfrac{\d x}{\d x}\big) H_1 \, \d x = 1 \cdot H_1(1) - 0 \cdot H_1(0) - \int_0^1 x \, \tfrac{\d H_1}{\d x} \, \d x \\
&\stackrel{\eqref{bc_h0h1_n1}}{=}& - \tfrac 1 2 \int_0^1 \big(\tfrac{\d}{\d x} (x^2-1)\big) \tfrac{\d H_1}{\d x} \, \d x = - 0 \cdot \tfrac{\d H_1}{\d x}(1) + \tfrac 1 2 \tfrac{\d H_1}{\d x}(0) + \tfrac 1 2 \int_0^1 (x^2-1) \, \tfrac{\d^2 H_1}{\d x^2} \, \d x \\
&\stackrel{\eqref{bc_h0h1_n1}}{=}& \tfrac 1 6 \int_0^1 \big(\tfrac{\d}{\d x} (x^3 - 3 x + 2)\big) \tfrac{\d^2 H_1}{\d x^2} \, \d x \\
&=& 0 \cdot \tfrac{\d^2 H_1}{\d x^2}(1) - \tfrac 1 3 \tfrac{\d^2 H_1}{\d x^2}(0) - \tfrac 1 6 \int_0^1 (x^3-3x+2) \, \tfrac{\d^3 H_1}{\d x^3} \, \d x,
\end{eqnarray*}
giving with \eqref{h1_gen_n_strong},
\begin{equation}\label{cond_d2h1_gen_n}
\tfrac{\d^2 H_1}{\d x^2}(0) = - C_1, \qquad C_1 := \tfrac12 \int_0^1 x (x^3-3x+2) (1-x^2)^{1-n} \, \d x.
\end{equation}
Thus we can integrate \eqref{h1_gen_n_strong} as follows:
\begin{eqnarray*}
\tfrac{\d^2 H_1}{\d x^2}(x) &\stackrel{\eqref{h1_gen_n_strong}, \eqref{cond_d2h1_gen_n}}{=}&
-C_1+\int_0^x x_1 \, (1-x_1^2)^{1-n} \, \d x_1, \\
\tfrac{\d H_1}{\d x}(x) &\stackrel{\eqref{bc_h0h1_n1}}{=}&
-C_1 x +\int_0^x x_1 \, (x-x_1) \, (1-x_1^2)^{1-n} \, \d x_1,
\end{eqnarray*}
so that with \eqref{bc_h0h1_n1},
\begin{align}
    H_1(x) &= \tfrac{C_1}{2} (1-x^2) - \int_x^1 \int_0^{x_1} x_2 \, (x_1-x_2) \, (1-x_2^2)^{1-n} \, \d x_2 \, \d x_1 && \mbox{for} \quad x \in [-1,1]. \label{h1_gen_n_strong_sol}
\end{align}
For $n = 1$ we have $C_1=\frac1{10}$ and this solution reads
\begin{align}\label{h1_n1_strong_sol}
    H_1(x) &= \tfrac1{20} (1-x^2) -\tfrac1{24}(1-x^4) =\tfrac{1}{120}(1-x^2)(1-5x^2) && \mbox{for} \quad x \in [-1,1].
\end{align}
For $n=2$ we get $C_1=\tfrac{1}{6}(5-\log 2)$ and the correction
\begin{align}
\label{h1_n2_strong_sol}
H_1(x)=\tfrac{C_1}{2}(1-x^2)-\tfrac14(3-3x^2-2\log 4+(x-1)^2\log(1-x)+(x+1)^2\log(1+x)).
\end{align}
Next, we compute $s_0$ and $\omega$. With help of \eqref{h0_gen_n_strong} and \eqref{h1_gen_n_strong}, \eqref{re_next_gen_n_strong} takes the form
\begin{equation}\label{re_next_n1_strong_bis}
3^2  =s_0^{4 - \alpha (n+3)} \omega^\alpha \left(\tfrac32\right)^{\alpha(n-1)} \quad\iff\quad \omega= 3^\frac{2}{\alpha} \left(\tfrac23\right)^{n-1} s_0^{n+3 -\frac{4}{\alpha}}.
\end{equation}
Thus we obtain
\[
\dot s_0 \stackrel{\eqref{ome_gen_n}}= \left(\tfrac32\right)^{n-1} \omega s_0^{-n-3} \stackrel{(\ref{re_next_n1_strong_bis})}= 3^{\frac2\alpha} s_0^{-\frac{4}{\alpha}},
\]
so that by normalizing $s_0(0) = 0$ through a time shift,
\begin{equation}\label{s0_gen_n}
s_0(t)= 3^{\frac{1-\gamma}{2}}\gamma^{-\gamma}t^\gamma , \qquad \gamma=\tfrac{\alpha}{\alpha+4}<\tfrac{1}{n+4}
\end{equation}
because $\alpha < \frac{4}{n+3}$, and consequently
\begin{equation}\label{om_gen_n_strong_2}
\omega(t) \stackrel{(\ref{re_next_n1_strong_bis})}
%= 3^{\frac2\alpha} (s_0(t))^{-\frac{4-\alpha(n+3)}{\alpha}}
= 3^\frac{1-\gamma}{2\gamma}\left(\tfrac23\right)^{n-1} (s_0(t))^{-\frac{1-\gamma (n+4)}{\gamma}}= \gamma^{1-\gamma(n+4)} 2^{n-1} 3^\frac{6-n-(n+4)\gamma}{2}
%3^\frac{(n+4)(1-\gamma)}{2} \left(\tfrac23\right)^{n-1}
t^{-(1-\gamma(n+4))}.
\end{equation}
Finally, the correction $s_1$ \ch{is} obtained from the next-to-leading order terms in equation \eqref{app_tfe_1}:
\begin{equation}\label{s1_gen_n}
s_1^{n+3} \dot s_1 = O\big(\omega^2\big) \stackrel{\eqref{om_gen_n_strong_2}}{=} O\big(t^{-2(1-\gamma(n+4))}\big) \quad \Leftrightarrow \quad s_1^{n+4} =\left\{\begin{array}{ll} O(t^{2(n+4)\gamma-1}) & \mbox{ if } \gamma\ne \frac1{2 (n+4)}, \\ O(\log t) & \mbox{ if } \gamma= \frac1{2(n+4)}.\end{array}\right.
\end{equation}
Since
\begin{equation}\label{exps}
s =(s_0^{n+4} + s_1^{n+4}(1+o(1)))^{\frac1{n+4}} = s_0\big(1+O(s_1^{n+4} s_0^{-n-4})\big),
\end{equation}
we have
\begin{equation}\label{s1.gen_n.s.f}
s(t) \stackrel{\eqref{s0_gen_n}, \eqref{s1_gen_n}}{=} 3^{\frac{1-\gamma}{2}} \gamma^{-\gamma} t^\gamma \begin{cases} \big(1+O(t^{-(1-\gamma(n+4))})\big) & \text{ if } \gamma \ne \frac{1}{2(n+4)}, \\ \big(1+O(t^{- \frac 1 2} \log t)\big) & \text{ if } \gamma = \frac{1}{2(n+4)}.\end{cases}
\end{equation}
The combination of \eqref{h0_gen_n_strong}, \eqref{h1_gen_n_strong_sol}, and \eqref{om_gen_n_strong_2} in \eqref{ansatz_self_n1_strong} yields at leading order as $s \to \infty$ (or equivalently $t\to \infty$) the quasi-self-similar profile $H$ according to
\begin{align} \nonumber
H(t,x) &= \tfrac32 (1-x^2) + C_2  \, t^{-(1-\gamma(n+4))}H_1(x) + O\big(t^{-2(1-\gamma(n+4))}\big) && \mbox{for} \quad x \in [-1,1],
\\
C_2 &= \gamma^{1-\gamma(n+4)} \, 2^{n-1} \, 3^\frac{6-n-(n+4)\gamma}{2} \label{def-C2}
\end{align}
where $\gamma = \frac{\alpha}{\alpha+4}$ and $H_1$ is defined in \eqref{cond_d2h1_gen_n}-\eqref{h1_gen_n_strong_sol}. For $n = 1$ this equation reduces to
\begin{equation}\label{strong-alt-s}
H(t,x) = \tfrac32 (1-x^2)
+ \tfrac{3^\frac{5(1-\gamma)}{2} \gamma^{1-5\gamma}}{120}\, t^{-(1-5\gamma)} \, (1-x^2)(1-5x^2)
+ O\big(t^{-2(1-5\gamma)}\big)
\end{equation}
for $x \in [-1,1]$.

%%%%%%%%%%%%%%%%%%%%
\subsection{The case $\alpha>\frac{4}{n+3}$: weak contact-line friction}\label{qss:weak}
%%%%%%%%%%%%%%%%%%%%

On assuming $\alpha > \frac{4}{n+3}$, separation of variables in \eqref{app_tfe_1} yields, at leading order for $s_0\gg 1$,
\begin{equation}\label{s_ode_gen_n_weak}
\dot s_0 = B_0^2 s_0^{-(n+3)}
\end{equation}
and
\begin{align}\label{bulk_lead_gen_n_weak}
H_0^{n-1} \tfrac{\d^3 H_0}{\d x^3} &= B^2_0 x && \mbox{in} \quad (0,1)
\end{align}
for some unknown constant $B_0 > 0$, whilst the condition \eqref{app_tfe_34} at the contact line yields
\begin{align}
    \tfrac{\d H_0}{\d x} &= 0 && \mbox{at} \quad x = 1. \label{re_gen_n_weak_1}
\end{align}
The system for $H_0$ is complemented with the boundary conditions \eqref{bc_h0h1_n1} and the mass constraint \eqref{mass_n1}, hence it coincides with \eqref{ss-new} with $D=0$. Therefore
\begin{equation}\label{def-BH0}
\mbox{$(B_0,H_0)$ is the unique solution to \eqref{ss-new} with $D=0$},
\end{equation}
(cf.~Theorem~\ref{thm} and \cite[Theorem~1.2]{BPW}), \ch{i.e. $H_0=H_{D=0}$} is the unique exact self-similar solution of \eqref{TFE-n} with zero contact angle and unit mass.
Integrating \eqref{s_ode_gen_n_weak} \ch{with normalization} $s_0(0) = 0$ then yields
\begin{equation}\label{qwe2_gen_n}
s_0(t) = \big((n+4) B_0^2 t\big)^{\frac{1}{n+4}}.
\end{equation}

\begin{Remark}{\rm
It is to be noted that the leading-order expansion given by $s_0$ and $H_0$ is self-consistent. This is in contrast to the case of strong contact-line friction, where the leading-order expansions in the bulk and at the contact line were incompatible with each other (cf.~\eqref{re_lead_gen_n_strong}). However, since we are interested in quantifying the correction coming from the contact-line frictional law, we shall be looking at $\omega H_1$ and $s_1$ in this case, too.
}\end{Remark}

At next-to-leading order for $s_0\gg 1$, the contact-line condition \eqref{app_tfe_34} implies
\begin{align}
    \tfrac{\d H_1}{\d x} & = - \big(H_0^{n-1} \tfrac{\d^3 H_0}{\d x^3}\big)^{\frac \alpha 2} \ \stackrel{\eqref{bulk_lead_gen_n_weak}}= - B_0^{-\alpha} && \mbox{at} \quad x = 1,  \label{re_gen_n_weak_2} \\
    \omega               & = s_0^{-\frac{\alpha(n+3)-4}{2}} \ \stackrel{(\ref{qwe2_gen_n})}= \big((n+4) B_0^2 t\big)^{- \frac{\alpha (n+3)-4}{2(n+4)}},\label{om_0_gen_n_weak}
\end{align}
where we have normalized $\omega$ conveniently since only the product $\omega H_1$ enters the expansion in \eqref{ansatz_self_n1_strong}. Furthermore, we have used that $\omega \tfrac{\d H_1}{\d x}$ has to have (strictly) negative sign at $x=1$ for the expansion to make sense around $x=1$ and be nontrivial.
In the bulk, the next-to-leading order terms in \eqref{app_tfe_1},
\begin{align*}
\omega \big((n-1) H_0^{n-2} \tfrac{\d^3 H_0}{\d x^3} H_1 + H_0^{n-1} \tfrac{\d^3 H_1}{\d x^3}\big) &= s_1^{n+3} \dot s_1 x && \mbox{for} \quad x \in (0,1),
\end{align*}
yield by separation of variables with \eqref{bulk_lead_gen_n_weak}
\begin{equation}
\label{s1_gen_n_w}
\omega = - C_2^{-1} B_0^{\frac{2(1-n)}{n}} s_1^{n+3} \dot s_1
\end{equation}
and
\begin{align}
\label{h1_gen_n_w}
(n-1) x H_1 + f^n \tfrac{\d^3 H_1}{\d x^3} &= - C_2 \, x f && \mbox{in} \quad (0,1)
\end{align}
for some $C_2 > 0$ without loss of generality (since the expansion only depends on $\omega H_1$), where we have defined
\begin{equation}\label{h0_c1_gen_n_weak}
f= B_0^{-\frac 2 n} H_0.
\end{equation}
%We integrate
\ch{Integrating} \eqref{s1_gen_n_w} using \eqref{om_0_gen_n_weak} \ch{we obtain} %: from
\[
\tfrac{\d s_1^{n+4}}{\d t} = O(\omega) = O\big( t^{\frac{4-\alpha (n+3)}{2(n+4)}}\big) \quad \Rightarrow\quad s_1^{n+4}=\left\{\begin{array}{ll} O\big(t^{\frac{2(n+6)-\alpha(n+3)}{2(n+4)}}\big) & \mbox{if } \ \alpha\ne 2 \frac{n+6}{n+3}, \\ O(\log t) & \mbox{if }\ \alpha=2 \frac{n+6}{n+3},
\end{array}\right.
\]
%
%we get
\ch{and} with \eqref{exps} and \eqref{qwe2_gen_n} \ch{this produces}
\begin{equation}\label{bulk1.1_gen_n}
s(t)=  \big((n+4) B_0^2 t\big)^{\frac{1}{n+4}} \begin{cases} \big(1+O\big(t^{- \frac{\alpha (n+3) - 4}{2(n+4)}}\big)\big) & \text{if } \alpha \ne 2 \frac{n+6}{n+3}, \\ \left(1+O(t^{-1} \log t)\right) & \text{if } \alpha = 2 \frac{n+6}{n+3}.\end{cases}
\end{equation}
Next, we turn our attention to determining $H_1$ and $C_2$. We distinguish three cases.

\bigskip

{\boldmath{$n=1$.}} For $n = 1$,  an explicit integration gives (cf. \eqref{self_exact}, \eqref{conj1}, and \eqref{h0_c1_gen_n_weak})
\begin{equation}\label{h0_n1_weak}
f(x) = \tfrac{1}{24} (1-x^2)^2, \quad
B_0^2 = 45 \quad \mbox{and} \quad H_0 = \tfrac{15}{8} (1-x^2)^2 \quad \mbox{for} \quad x \in [-1,1],
\end{equation}
whence $s_0(t) \stackrel{\eqref{qwe2_gen_n}}{=} (225 t)^{\frac 1 5}$, and \eqref{om_0_gen_n_weak} and \eqref{bulk1.1_gen_n} reduce to
\begin{equation}\label{bulk1.1}
\omega = (225 t)^{- \frac{2 (\alpha-1)}{5}}, \qquad s(t)=  (225\, t)^\frac15 \begin{cases} \big(1+O\big(t^{\frac{-2(\alpha-1)}{5}}\big)\big) & \text{if } \alpha \ne \frac 7 2, \\ \big(1+O(t^{-1} \log t)\big) & \text{if } \alpha = \frac 7 2.\end{cases}
\end{equation}
We integrate \eqref{h1_gen_n_w} %\eqref{h1_n1_w}
using the four conditions \eqref{bc_h0h1_n1}, \eqref{mass_n1}, and \eqref{re_gen_n_weak_2}, leading to $C_2= (45)^{\frac\alpha 2}$ and
\begin{equation}\label{def-H1-w}
H_1=  \tfrac{(45)^{\frac\alpha 2}}{6} \big(\tfrac{1}{5}(1-x^2)-\tfrac{1}{4} (1-x^2)^2\big)=-\tfrac{(45)^{\frac\alpha 2}}{120}(1-x^2)(1-5x^2).
\end{equation}
Collecting \eqref{h0_n1_weak}, \eqref{bulk1.1} and \eqref{def-H1-w} in \eqref{ansatz_self_n1_strong}, at leading order as $s \to \infty$ (or equivalently $t\to \infty$)
the approximate position of the free boundary is given by \eqref{bulk1.1}, and the quasi-self-similar profile $H$ is given by
\[
H(t,x) = \tfrac{15}{8} (1-x^2)^2 - \tfrac18 3^{\frac{\alpha-1}{5}} 5^{- \frac{3\alpha+2}{10}}  t^{- \frac{2 (\alpha-1)}{5}} (1-x^2) (1-5x^2) + O\big(t^{-\frac{4(\alpha-1)}{5}}\big).
\]
\bigskip

{\boldmath{$n\in (1,\frac32)$.}} We now consider $n > 1$. Defining $g$ as
\begin{equation}\label{decomp_h1_fg}
H_1 = \ch{-} \tfrac{C_2}{n} f + g,
\end{equation}
we get from \eqref{h1_gen_n_w}
\begin{subequations}\label{problem_g}
\begin{align}\label{ode_g}
(n-1) x g + f^n \tfrac{\d^3 g}{\d x^3} &= 0 && \mbox{for} \quad x \in (0,1),
\end{align}
and the boundary conditions \eqref{bc_h0h1_n1} and \eqref{re_gen_n_weak_2} transform into
\begin{align}
g &= 0 && \mbox{at} \quad x = 1, \label{bc_g_x1}\\
\tfrac{\d g}{\d x} &= - B_0^{\alpha} && \mbox{at} \quad x = 1, \label{bc_dg_x1}\\
\tfrac{\d g}{\d x} &= 0 && \mbox{at} \quad x = 0.
\end{align}
\end{subequations}
Once $g$ is determined, the mass constraint \eqref{mass_n1} leads to
\begin{equation}\label{c2_gen_n_weak}
C_2 = - \frac{n \int_0^1 g(x) \, \d x}{\int_0^1 f(x) \, \d x}.
\end{equation}
In order to determine solvability of \eqref{problem_g}, we use \cite[Theorem~1.3]{BPW}, that is,
\begin{equation}\label{asym_f}
f = \begin{cases} C_3 (1-x)^2 (1+o(1)) & \text{for } 0 < n < \frac 3 2, \\
C_4 (1-x)^2 \left(-\log(1-x)\right)^{\frac 2 3} (1+o(1)) & \text{for } n = \frac 3 2, \\
C_5 (1-x)^{\frac 3 n} (1+o(1)) & \text{for } \frac 3 2 < n < 3, \end{cases} \qquad \mbox{as $x \nearrow 1$,}
\end{equation}
where $C_3, C_4, C_5 > 0$ only depend on $n$. The asymptotics \eqref{asym_f} imply that \eqref{problem_g} has no solution for $n \in \left[\frac 3 2,3\right)$: indeed, for $n \in (\frac 3 2, 3)$ from \eqref{ode_g}\ch{-}\eqref{bc_dg_x1} and \eqref{asym_f} we infer
\[
\tfrac{\d^3 g}{\d x^3} = \begin{cases} -\tfrac12 B_0^{\alpha} C_4^{-\frac32} \,  (1-x)^{-2} (-\log(1-x))^{-1} & \text{for } n = \frac 3 2,
\\[1ex]
-(n-1)B_0^{\alpha} C_5^{-n} (1-x)^{-2} (1+o(1)) & \text{for } n \in (\frac 3 2,3)\ch{,} \end{cases} \quad \mbox{as $x \nearrow 1$},
\]
in contradiction  with \eqref{bc_dg_x1}.

For $n \in \left[1,\frac 3 2\right)$, we apply a further splitting according to
\begin{equation}\label{decomp_g_vw}
g = \tfrac 1 2 B_0^{\alpha} (1-x^2) + v + w,
\end{equation}
where
\begin{subequations}\label{problem_v}
\begin{align}
f^n \tfrac{\d^3 v}{\d x^3} &= \tfrac{1-n}{2} B_0^{\alpha} x \, (1-x^2) && \mbox{for} \quad x \in (0,1), \label{ode_v} \\
v &= \tfrac{\d v}{\d x} = 0 && \mbox{at} \quad x = 1, \label{bc_v_1} \\
\tfrac{\d v}{\d x} &= 0 && \mbox{at} \quad x = 0, \label{bc_v_2}
\end{align}
\end{subequations}
and
\begin{subequations}\label{problem_w}
\begin{align}
\LL w := (n-1) \, x \, w + f^n \tfrac{\d^3 w}{\d x^3} &= (1-n) \, x \, v && \mbox{for} \quad x \in (0,1), \label{ode_w} \\
w &= \tfrac{\d w}{\d x} = 0 && \mbox{at} \quad x = 1, \label{bc_w_1} \\
\tfrac{\d w}{\d x} &= 0 && \mbox{at} \quad x = 0. \label{bc_w_2}
\end{align}
\end{subequations}
We first construct a solution to \eqref{problem_v}. We have
\begin{align*}
\tfrac{\d^2 v}{\d x^2} \ \ &\stackrel{\mathclap{\eqref{ode_v}}}{=} \ \ - \tfrac{n-1}{2} B_0^{\alpha} \int_0^x \frac{x_1 (1-x_1^2)}{(f(x_1))^n} \, \d x_1 + C_6, \\
\tfrac{\d v}{\d x} \ \ &\stackrel{\mathclap{\eqref{bc_v_2}}}{=} \ \ - \tfrac{n-1}{2} B_0^{\alpha} \int_0^x \int_0^{x_1} \frac{x_2 (1-x_2^2)}{(f(x_2))^n} \, \d x_2 \, \d x_1 + C_6 x \\
&= - \tfrac{n-1}{2} B_0^{\alpha} \int_0^x \frac{x_2 (x-x_2) (1-x_1^2)}{(f(x_2))^n} \, \d x_2 + C_6 x, \\
C_6 \ \ &\stackrel{\mathclap{\eqref{bc_v_1}}}{=} \ \ \tfrac{n-1}{2} B_0^{\alpha} \int_0^1 \frac{x (1-x)^2 (1+x)}{(f(x))^n} \, \d x,
\end{align*}
so that
\begin{align}
v \stackrel{\eqref{bc_v_1}}{=} \tfrac{n-1}{2} B_0^{\alpha} \int_x^1 \int_0^{x_1} \frac{x_2 (x-x_2) (1-x_1^2)}{(f(x_2))^n} \, \d x_2 \, \d x_1 - \tfrac{C_6}{2} (1-x^2). \label{sol_v}
\end{align}
%
%Notbaly,
\ch{Notably}, in view of \eqref{asym_f}, \eqref{sol_v} only yields a well-defined solution for $n \ch{\in [1,\frac 3 2)}$. Lastly, in order to find the solution of \eqref{problem_w}, we use
\begin{align*}
\int_0^1 x \, f^{-n} \, w \, (\LL w) \, \d x \qquad & = \qquad (n-1) \int_0^1 x^2 \, f^{-n} \, w^2 \, \d x + \int_0^1 x \, w \, \tfrac{\d^3 w}{\d x^3} \, \d x
\\ &
 \stackrel{\mathclap{\eqref{bc_w_1}}}{=} \qquad (n-1) \int_0^1 x^2 \, f^{-n} \, w^2 \, \d x - \int_0^1 w \, \tfrac{\d^2 w}{\d x^2} \, \d x - \tfrac 1 2 \int_0^1 x \, \tfrac{\d}{\d x} \big(\tfrac{\d w}{\d x}\big)^2 \d x
 \\ & \stackrel{\mathclap{\eqref{bc_w_1}, \eqref{bc_w_2}}}{=} \qquad (n-1) \int_0^1 x^2 \, f^{-n} \, w^2 \, \d x + \tfrac 3 2 \int_0^1 \big(\tfrac{\d w}{\d x}\big)^2 \d x > 0 \quad\mbox{for all} \quad w\not \equiv 0.
\end{align*}
This proves coercivity for $n \in \left[1,\frac 3 2\right)$, so that the Lax-Milgram theorem yields existence of a unique solution to \eqref{problem_w}. Together with \eqref{sol_v}, this yields existence of a unique solution $g$ to \eqref{problem_g}, hence of a unique $H_1$.

\bigskip

{\boldmath{$n\in (\frac32,3)$.}} The afore-mentioned non-existence of solutions to \eqref{problem_g} for mobility exponents $n \in [\frac 3 2,3)$ entails that the ansatz \eqref{ansatz_self_n1_strong} breaks down for this range of mobilities in the regime of strong contact-line friction $\alpha > \frac{4}{n+3}$. This necessitates a matched-asymptotics approach in which we distinguish between outer and inner region. Since the case $n = \frac 3 2$ contains a logarithmic resonance (see \eqref{asym_f}), we only concentrate on the range $n \in (\frac 3 2,3)$ in what follows.

\medskip

With $H_0$ and $s_0$ given by \eqref{def-BH0} and \eqref{qwe2_gen_n}, respectively, it follows from the previous analysis that $s = s_0 (1+o(1))$ and $H(t,x)=H_0(x)$ in the outer region (whose extent is to be determined). Near $x = 1$, from \eqref{h0_c1_gen_n_weak} and \eqref{asym_f}, we infer for $n \in (\frac 32,3)$ that
\begin{align}\label{outer_contact}
H_0 &= C_5 B_0^{\frac 2 n} (1-x)^{\frac 3 n} (1+o(1))
&& \mbox{for} \quad 0 < 1-x \ll 1,
\end{align}
where more precisely
\begin{equation}\label{def-c5}
C_5 = \big(\tfrac 3 n (\tfrac 3 n - 1) (2 - \tfrac 3 n)\big)^{-\frac 1 n}.
\end{equation}
In the inner region $x \nearrow 1$, we use the traveling-wave ansatz
\begin{equation}\label{inner_tw}
H(t,x)= s F_\mathrm{in}(\xi) \quad \xi = s(1-x).
\end{equation}
Inserted into \eqref{tfe_classical_1_fixed} with $s = s_0 (1+o(1))$, this gives, after an integration using \eqref{tfe_classical_2_fixed} and \eqref{tfe_classical_4_fixed},
\begin{subequations}\label{inner_sys}
\begin{align}\label{eq_inner}
F_\mathrm{in}^{n-1} \tfrac{\d^3 F_\mathrm{in}}{\d \xi^3} &= - \dot s_0 && \text{for} \quad \xi > 0.
\end{align}
The boundary condition \eqref{tfe_classical_4_fixed} is then automatically satisfied; the conditions \eqref{tfe_classical_2_fixed} and \eqref{tfe_classical_3_fixed} translate into
\begin{align}
F_\mathrm{in} = 0 \qquad\mbox{and}\qquad  \big(\tfrac{\d F_\mathrm{in}}{\d\xi}\big)^2 = (\dot s_0)^\alpha && \text{at} \quad \xi = 0, \label{inner_h_contact}
\end{align}
and are complemented with the condition
\begin{align}\label{inner_d2h_bulk}
\tfrac{\d^2 F_\mathrm{in}}{\d\xi^2} &\to 0 && \text{as} \quad \xi \to \infty \quad \text{for} \quad \tfrac 3 2 < n < 3,
\end{align}
\end{subequations}
which in view of \eqref{outer_contact} is required for the matching to the outer solution $H_0$. Well-posedness of \eqref{inner_sys} may be obtained using the strategies in \cite{CG-simai} or \cite{ggo.2016}, where related problems were considered. Here we sketch the application of the argument in \cite{ggo.2016} to the present case. For $\xi\ll 1$, a  one-parametric solution family to \eqref{eq_inner} and \eqref{inner_h_contact} is given by
\begin{align}\label{inner_xi_0}
F_\mathrm{in}(\xi) = \begin{cases} \big(\dot s_0^{\frac \alpha 2} \xi - \frac{\dot s_0^{1+\frac \alpha 2(1-n)}}{(4-n) (3-n) (2-n)} \xi^{4-n} + a_\mathrm{in} \xi^2 + \mathrm{h.o.t.}\big) & \text{for } n \ne 2 \\[1ex]
\dot s_0^{\frac\alpha 2} \big(\xi + \frac{\dot s_0^{1-\alpha} }{2} \xi^2 \log \xi + a_\mathrm{in} \xi^2 + \mathrm{h.o.t.}\big) & \text{for } n = 2 \end{cases}
\quad \mbox{as $\xi\searrow 0$,} \quad a_\mathrm{in}\in \R
\end{align}
where $\mathrm{h.o.t.}$ denotes higher-order terms. For  $\xi\gg 1$, a one-parametric (plus translation) solution family to \eqref{eq_inner} and \eqref{inner_d2h_bulk} is given by
\begin{align}\label{inner_xi_inf}
F_\mathrm{in} = \dot s_0^{\frac 1 n}\big(C_5 \xi^{\frac 3 n} + b_\mathrm{in} \xi + \mathrm{h.o.t.}\big) \quad \text{for $\tfrac 3 2 < n < 3$} \quad \text{as} \ \xi \to \infty, \quad b_\mathrm{in}\in \R.
\end{align}
Therefore, for $\frac 3 2 < n < 3$, \eqref{inner_xi_0} and \eqref{inner_xi_inf} define two two-dimensional solution manifolds (with parameters $(\xi,a_\mathrm{in})$ and $(\xi,b_\mathrm{in})$, respectively) of the three-dimensional dynamical system $(F_\mathrm{in}, \frac{\d F_\mathrm{in}}{\d \xi}, \frac{\d^2F_\mathrm{in}}{\d \xi^2})$: their intersection is the one-dimensional solution curve determining $a_\mathrm{in}$ and $b_\mathrm{in}$ and defining the solution to \eqref{inner_sys}.

In terms of $x$, expansion \eqref{inner_xi_inf} translates at leading order into
\begin{align}\label{in-out}
H(t,x) &\stackrel{\eqref{inner_tw}}= s_0 F_\mathrm{in} = \dot s_0^{\frac 1 n} s_0^{\frac{n+3}{n}} C_5 (1-x)^{\frac 3 n} (1+o(1)) && \mbox{for} \quad s^{-1} \ll 1-x\ll 1.
\end{align}
For $s\gg 1$, expansions \eqref{outer_contact} and \eqref{in-out} thus have an overlapping region,  $s^{-1}\ll 1-x \ll 1$: matching them yields
\[
\dot s_0^{\frac 1 n} s_0^{\frac{n+3}{n}} C_5 = B_0^{\frac 2 n} C_5,
\]
which is fulfilled because of \eqref{s_ode_gen_n_weak}. Since
\begin{align*}%\nonumber
s_0 F_\mathrm{in} &\stackrel{\eqref{inner_xi_0}}{=} s_0^2\dot s_0^{\frac \alpha 2} (1-x) (1+o(1)) \stackrel{(\ref{s_ode_gen_n_weak})}= B_0^\alpha s_0^\frac{4-\alpha(n+3)}{2}(1-x) (1+o(1)) \\
&\stackrel{(\ref{qwe2_gen_n})}= B_0^\alpha \big((n+4) B_0^2 t\big)^{\frac{4-\alpha(n+3)}{2(n+4)}}(1-x) (1+o(1)) && \mbox{for} \quad 0 \le 1-x \ll s^{-1}, %\label{in-in}
\end{align*}
we conclude that $H(t,x)= \tilde H_0(t,x)(1+o(1))$, where
\begin{equation}\label{match}
\tilde H_0(t,x)\sim \left\{\begin{array}{ll}
H_0(x) & \mbox{for $1-x \gg s^{-1}$}
\\ B_0^\alpha \big((n+4) B_0^2 t\big)^{-\beta} (1-x)& \mbox{for $1-x \ll s^{-1}$},
\end{array}\right. \quad s\gg 1, \quad \beta= \frac{\alpha(n+3)-4}{2(n+4)}>0.
\end{equation}

%%%%%%%%%%%%%%%%%%%%
\section{Conclusions \ch{and outlook}}\label{s:concl}
%%%%%%%%%%%%%%%%%%%%

We have obtained an asymptotic description of the long-time dynamics of \ch{solutions $(H,s)$ to} \eqref{tfe_classical_fixed}. \ch{This translates back to the original height via $h(t,y)=s^{-1}H(t,s^{-1}y)$ with support $(-s,s)$.}
% \ch{; we remind that $s$ denotes the droplet's support and $H=s^{-1}h(t,s^{-1}y)$, where $h$ denotes the droplet's height.

%\medskip

In the balanced case, $\alpha=\tfrac{4}{n+3}$, we have shown that \eqref{tfe_classical_fixed} admits for any $D\ge 0$ a unique self-similar profile, determined by \eqref{ss-new}, and $s$ scales like $t^\frac{1}{n+4}$. Both profile and speed depend on $D$, the profile ranging from the zero contact-angle one for $D=0$ to a parabolic shape as $D\to +\infty$. Numerical solutions to \eqref{ss-new} have been provided in Fig. \ref{fig:exact_n1_n2}.}

A very interesting question concerns stability, i.e., the convergence of solutions $H$ of \eqref{tfe_classical_fixed} to the self-similar profile \eqref{self_exact}. This issue has already been raised in \cite[\S6]{CarrilloToscani2002} and \cite[\S8]{Gnann2015} in a similar context and could be faced, starting from $n=1$, either by energy-entropy methods \cite{CarrilloToscani2002,carlen-ulusoy,MMCS,CarlenUlusoy2014}, by studying global-in-time classical solutions for perturbations of special solutions like a self-similar profile \cite{ggo.2013,Gnann2015,seis}, a traveling wave \cite{GnannIbrahimMasmoudi2019,ggko}, or an equilibrium-stationary solution \cite{GKO,BGKO2016,Knuepfer2011,Knuepfer2015,Knuepfer2022,Esselborn2016,MajdoubMasmoudiTayachi2021}. In the three latter cases, the difficulty lies in finding suitable estimates for the linearized evolution of perturbations. Unlike in the case of the Smyth-Hill profile \cite{SmythHill1988}, this linearization does not carry an apparent symmetric structure, which is why the linear analysis is presumably more involved. Fig.~\ref{fig:generic_solution_rescaled} provides numerical simulations supporting convergence \ch{in the balanced case for $n=1$ and $n=2$.}
%}\end{Remark}

\begin{figure}[t]
\centering
    \includegraphics[width=0.4\textwidth]{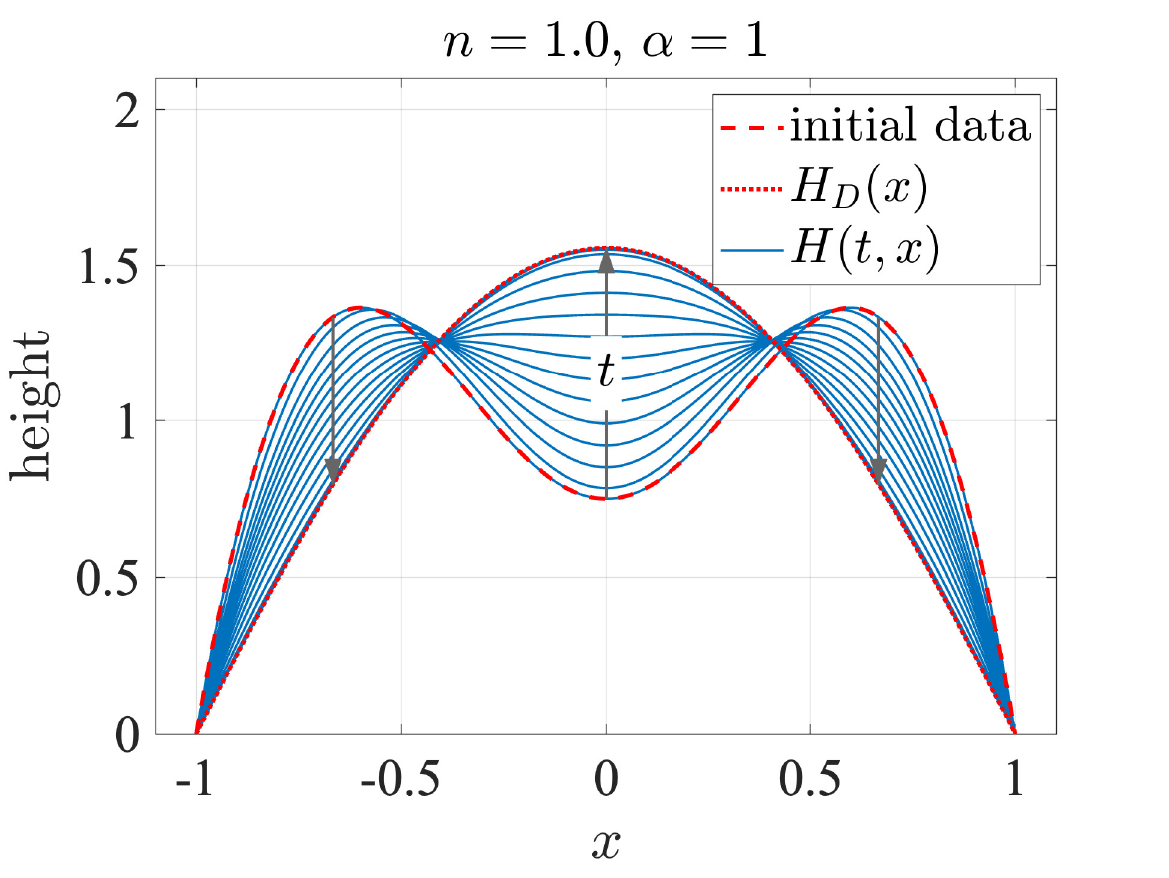}\qquad
    \includegraphics[width=0.4\textwidth]{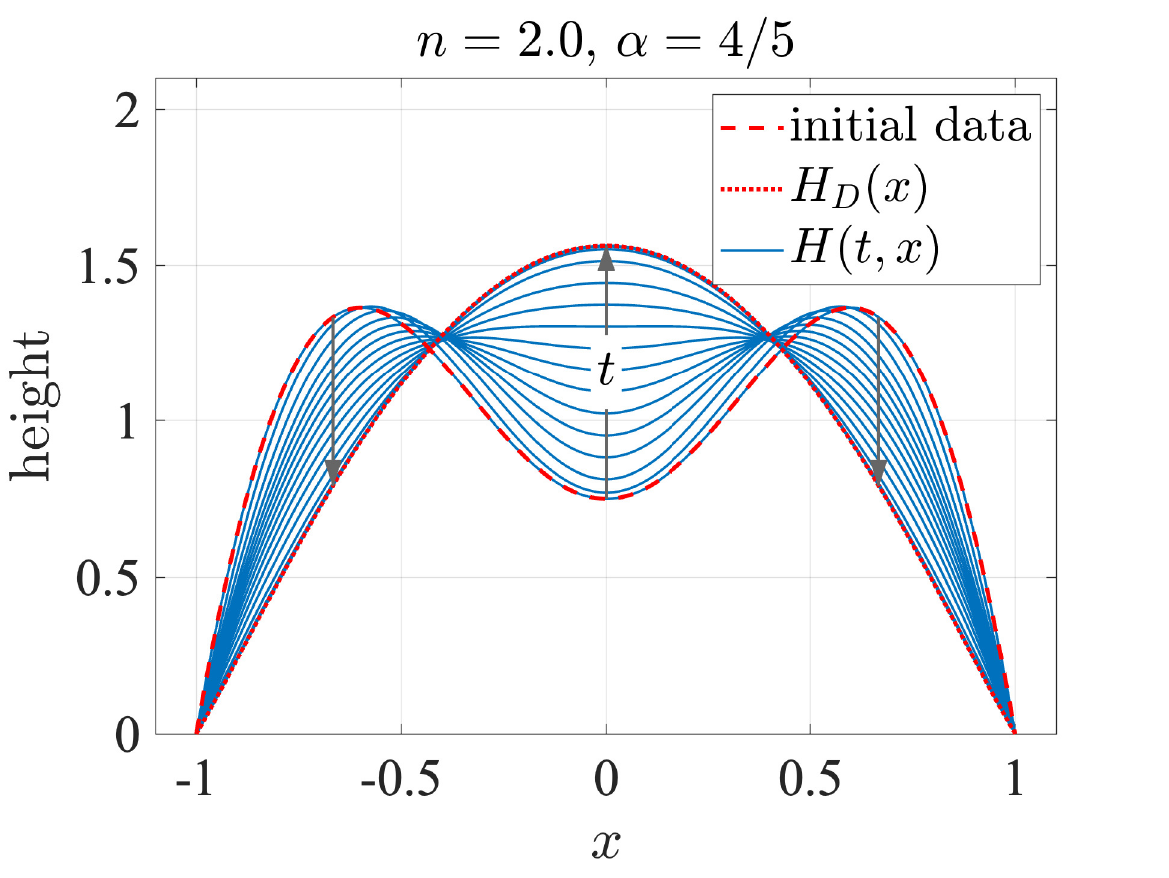}

    \caption{For $D=1$ and $n=1$ (left) or $n=2$ (right), the solution $H(t,x)$ to \eqref{tfe_classical_fixed} for increasing times \ch{$0\le t<\infty$} (blue lines, \ch{time increases along arrows}) with initial datum $H(t=0,x)=\tfrac{15}{4}((1+x)(1-x)-\tfrac25(1+\cos(\pi x))$ (red dashed line), compared with the exact self-similar solutions $H_D(x)$ to \eqref{ss-new} approached as $t\to\infty$ (red \ch{dotted} line).}
 \label{fig:generic_solution_rescaled}
\end{figure}

%\ch{I}n the non-self-similar case, $\alpha\ne \frac{4}{n+3}$.

\medskip

For strong contact-line friction, $\alpha<\frac{4}{n+3}$, $H$ and $s$ obey the asymptotic
%
%leading order asymptotic
%\begin{equation}\label{concl-lead-strong}
%H(t,x)=\tfrac32 (1-x^2) \qquad\mbox{and}\qquad s(t)= 3^{\frac{1-\gamma}{2}} \gamma^{-\gamma} t^\gamma,\qquad \gamma=\tfrac{\alpha}{\alpha+4}<\tfrac{1}{n+4}.
%\end{equation}
%
%
\begin{subequations}
\begin{align} \label{strong-alt-gen-n-s-bis}
H(t,x) &= \tfrac32 (1-x^2) + C_2  \, t^{-(1-\gamma(n+4))} H_1(x)  + O\big(t^{-2(1-\gamma(n+4))}\big) && \mbox{for} \quad x \in [-1,1], \\
s(t) &= 3^{\frac{1-\gamma}{2}} \gamma^{-\gamma} t^\gamma \begin{cases} \big(1+O(t^{-(1-\gamma(n+4))})\big) & \text{ if } \gamma \ne \frac{1}{2(n+4)}, \\ \big(1+O(t^{- \frac 1 2} \log t)\big) & \text{ if } \gamma = \frac{1}{2(n+4)},\end{cases} \label{s1.gen_n.s.f.bis}
\end{align}
\end{subequations}
where $H_1$ and $C_2$ are defined in \S \ref{qss:strong} (see \eqref{cond_d2h1_gen_n}-\eqref{h1_gen_n_strong_sol} and \eqref{def-C2}), and
\begin{align*}
H_1(x) &= \tfrac{1}{120}(1-x^2)(1-5x^2) && \text{if} \quad n=1.
\end{align*}
The leading-order profile is a parabolic one, with finite non-zero contact angle, and coincides with the exact self-similar profile in the limiting case $D=\infty$ (\S \ref{s:ss}). \ch{In other words, $H_D\to H_0$ as $D\to\infty$.} In addition, the evolution of the contact line, as given by \eqref{s1.gen_n.s.f.bis}, is slower than the standard one and is dominated by the contact-line frictional exponent $\alpha$. Therefore, for $\alpha<\frac{\ch{4}}{n+\ch{3}}$ contact-line friction dominates the long-time dynamics uniformly in space and time (Fig.~\ref{7Sb}, first column).
\begin{figure}[h]
    \centering
    \includegraphics[width=0.4\textwidth]{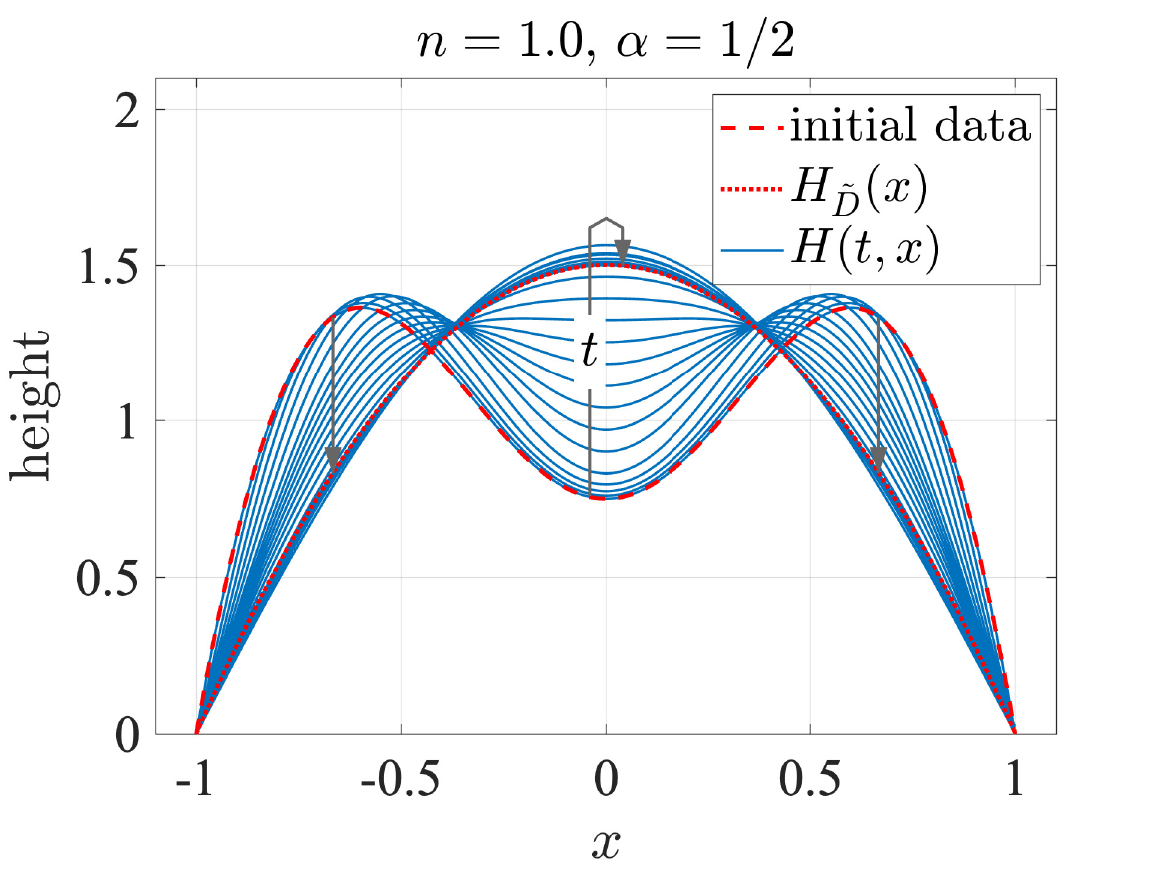}
    \includegraphics[width=0.4\textwidth]{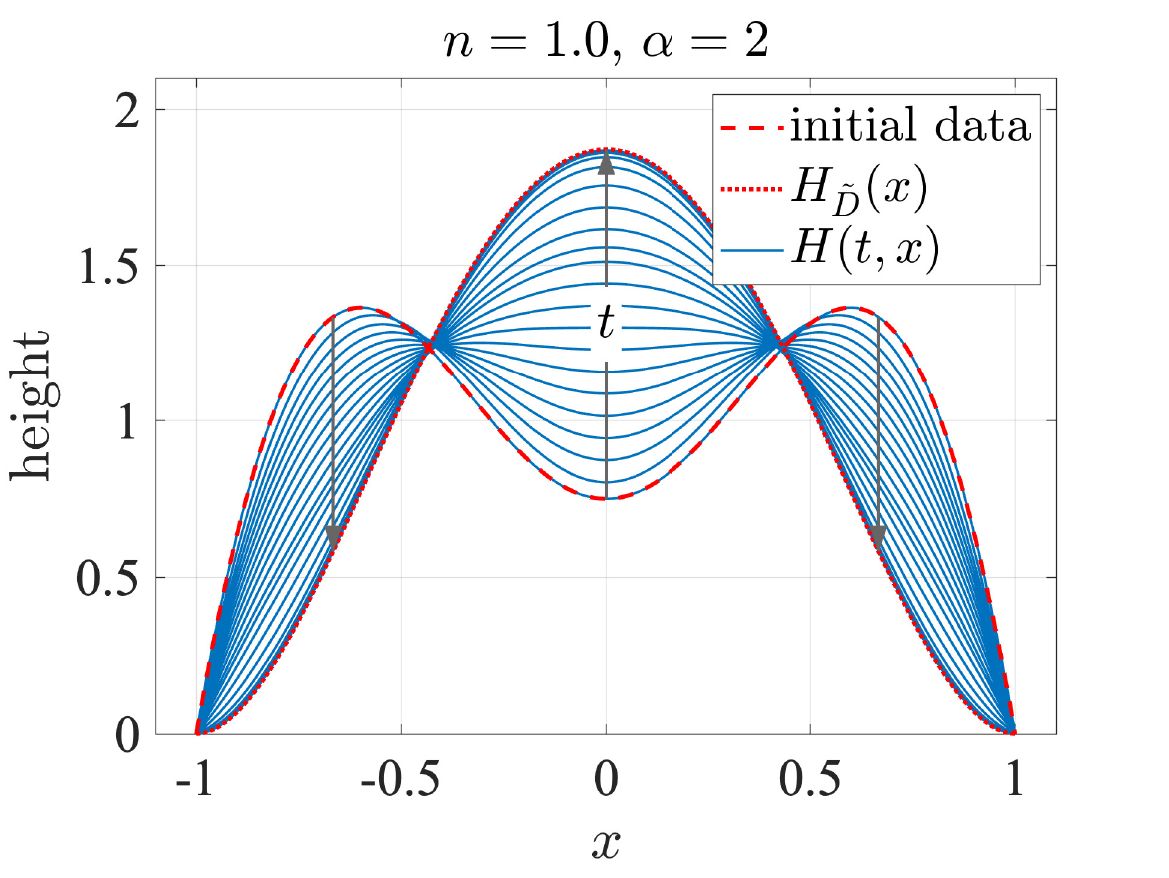}

    \includegraphics[width=0.4\textwidth]{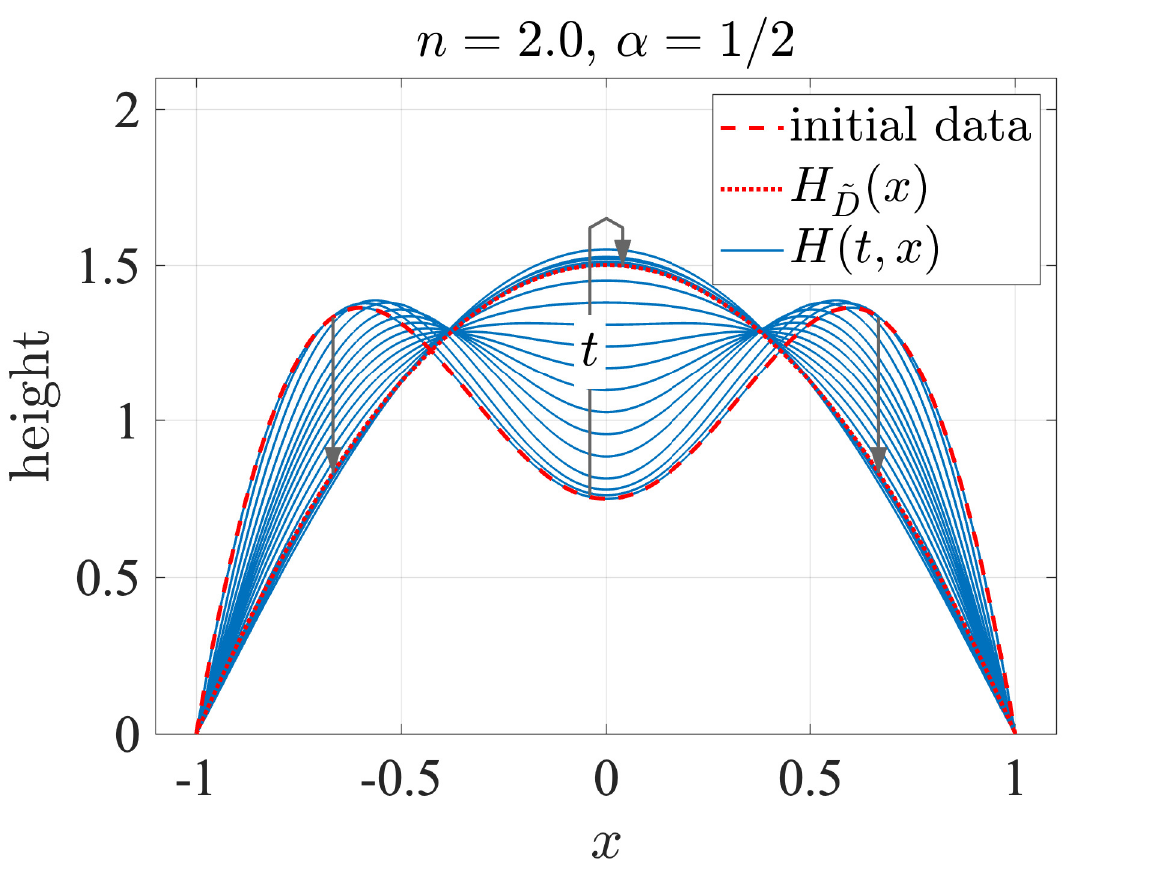}
    \includegraphics[width=0.4\textwidth]{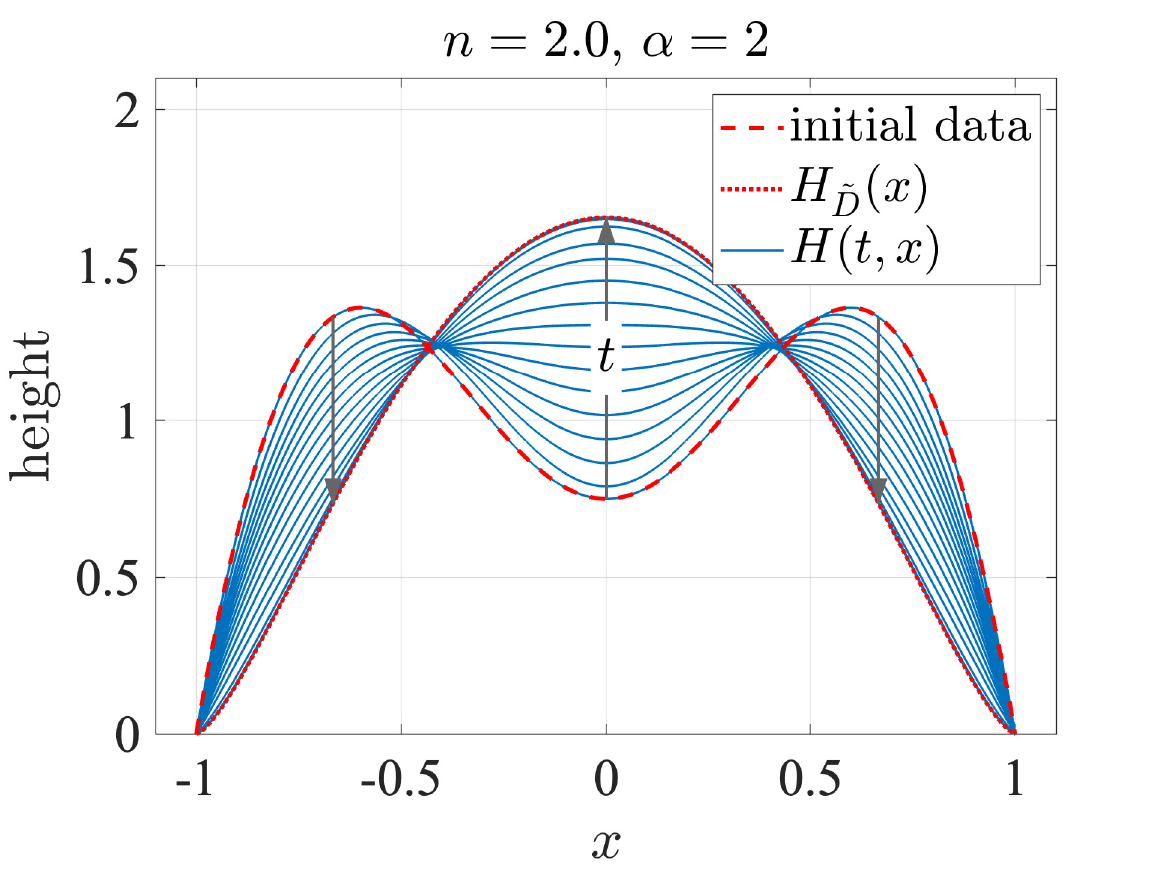}

    \caption{For $D=1$ and $n=1$ (top) or $n=2$ (bottom), the solution $H(t,x)$ to \eqref{tfe_classical_fixed} for different $t$ (blue lines) with initial datum $H(t=0,x)=\tfrac{15}{4}\big((1+x)(1-x)-\tfrac25(1+\cos(\pi x))\big)$ (red dashed line) for strong contact line friction $\alpha=1/2$ (left) and weak contact line friction $\alpha=2$ (right) compared with the exact self-similar solutions $\ch{H_{\tilde{D}}}(x)$ to \eqref{ss-new} approached as $t\to\infty$ (red \ch{dotted} line\ch{s, with $\tilde{D}\to\infty$ for $\alpha=1/2$, cf. Theorem \ref{thm}, and $\tilde{D}=0$ for $\alpha=2$}). Note that for $\alpha = \frac 1 2$ solutions go slightly above the graph of $\ch{H_{\tilde{D}}}$ before relaxing (a manifestation of the lack of \ch{a} comparison principle). \ch{The time evolution is indicated by dark gray arrows.}}\label{7Sb}
\end{figure}
Notably, the correction $H_1$, which is dictated by attaining the dynamical contact-line condition, is not localized near the contact line, but rather propagates throughout the solution's support (\ch{%\st{Fig.~}\ref{8-11S}\st{ and }\ref{9-10S}\st{, first column}
Fig.~\ref{fig:correction-h1}, middle and right}).
\begin{figure}[h]
    \centering
    \includegraphics[width=0.315\textwidth]{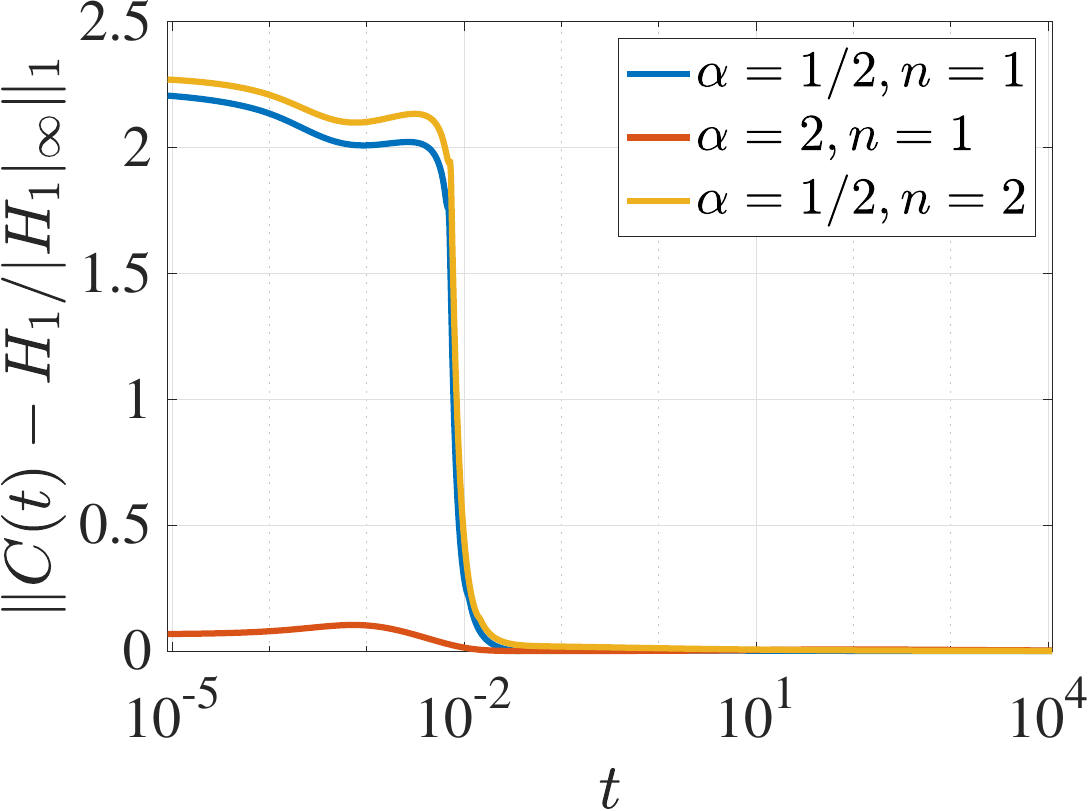}\hfill
    \includegraphics[width=0.335\textwidth]{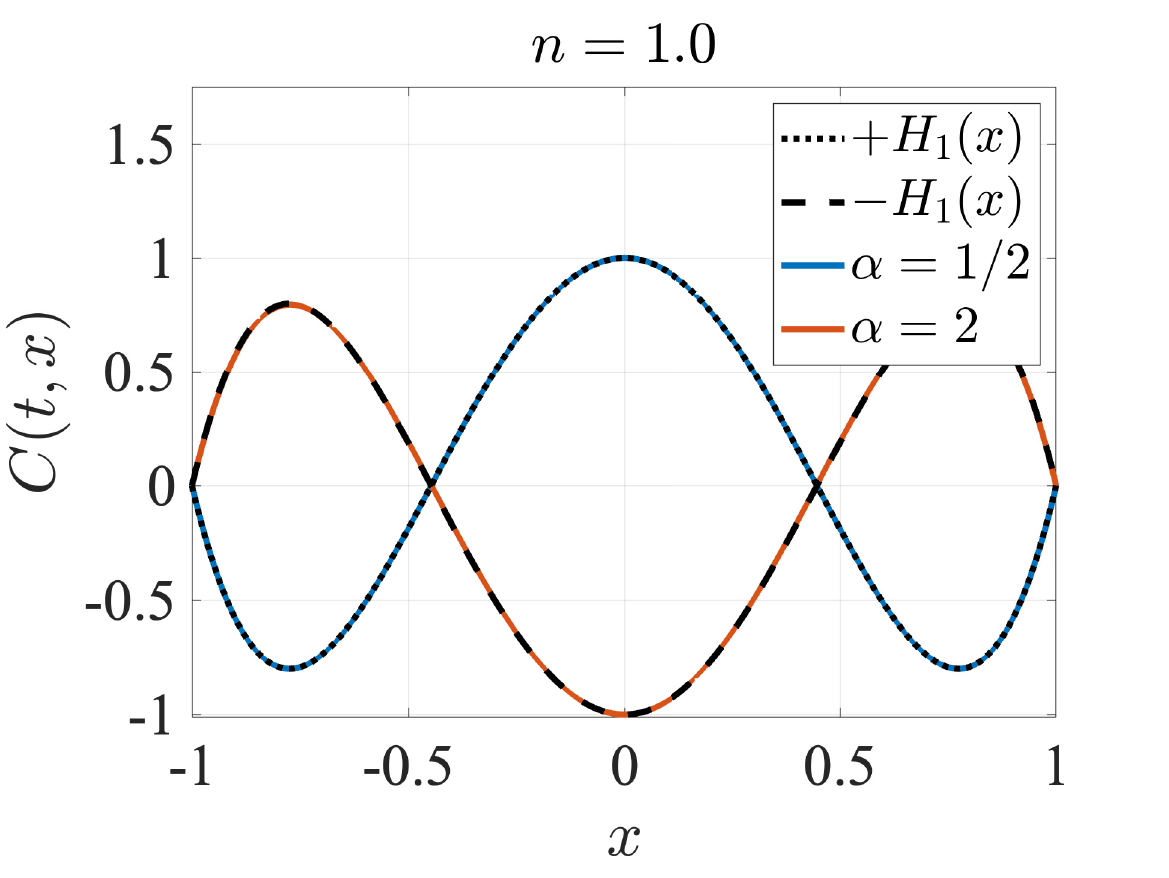}\hfill
    \includegraphics[width=0.335\textwidth]{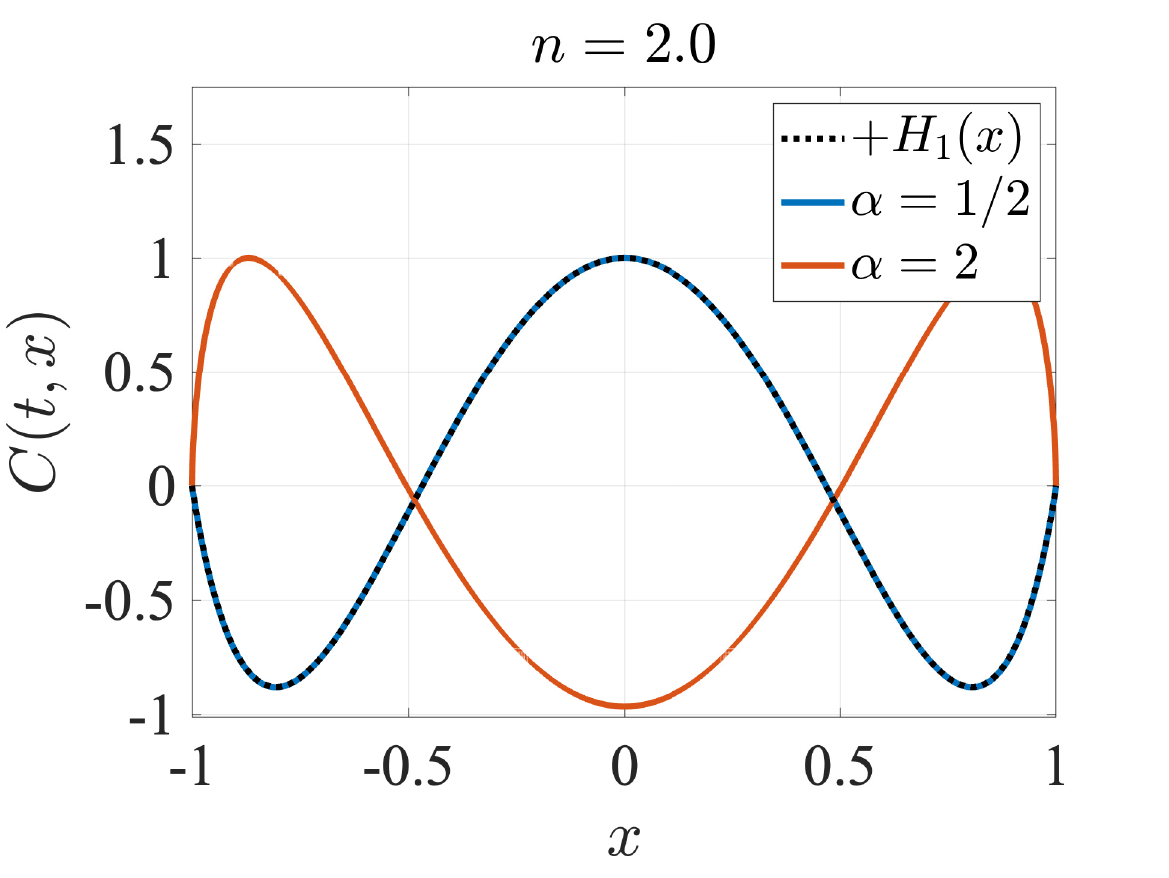}
    \caption{\ch{Convergence of the correction to the self-similar solution, where (left) we show the $L^1$ norm of the difference between $C(t,x)=\bigl(H(t,x)-H_0(x)\bigr)/|H(t,\cdot)-H_0(\cdot)|_\infty$ and the theoretically predicted, normalized correction $\pm H_1(x)/|H_1(\cdot)|_\infty$.
    For comparison, in (middle) and (right) we show $C(t,x)$ at $t=200$ for $n=1$, respectively $n=2$. Corrections $C$ (colored) are compared with theoretically predicted ones (dashed/dotted), when the latter are available.}}
    \label{fig:correction-h1}
\end{figure}

For weak contact-line friction, $\alpha>\frac{4}{n+3}$, the leading\ch{-}order dynamics instead coincide with the zero-contact-angle ones in terms of both profile and speed, in the sense that
\begin{equation}\label{concl-lead-weak}
H(t,x)=H_{0}(x) \qquad\mbox{and}\qquad s(t)= \big((n+4) B_0^2 t\big)^{\frac{1}{n+4}},
\end{equation}
where $(B_0,H_0)$ is the unique solution to \eqref{ss-new} with $D=0$. In other words, $H_0\ch{=H_{D=0}}$ is the unique self-similar profile of \eqref{TFE-n} with zero contact angle and \ch{normalized} mass (Fig. \ref{7Sb}, second column). \ch{O}ur results on the corrections turn out to depend on the mobility exponent $n$: if $n\in [1,3/2)$ we find a global estimate as above, in the sense that
\begin{subequations}\label{exp-weak-1_bulk1.1_gen_n_bis}
\begin{align}
H(t,x) &= H_0(x)+ \big((n+4) B_0^2 t\big)^{- \beta }H_1(x)+ O(t^{- 2\beta}), \qquad \beta= \frac{\alpha (n+3)-4}{2(n+4)}, \label{exp-weak-1} \\
s(t) &= \big((n+4) B_0^2 t\big)^{\frac{1}{n+4}} \begin{cases} \big(1+O(t^{-\beta})\big) & \text{if } \alpha \ne 2 \frac{n+6}{n+3}, \\ \big(1+O(t^{-1} \log t)\big) & \text{if } \alpha = 2 \frac{n+6}{n+3}\end{cases} \label{bulk1.1_gen_n_bis}
\end{align}
\end{subequations}
where $H_1(x)$ is uniquely determined in \S\ref{qss:weak}. On the other hand, if $n\in (3/2,3)$ we are only able to qualify a {\em local} correction:
$$
H(t,x) \sim \left\{\begin{array}{ll}
H_0(x) & \mbox{for $1-x \gg t^{-\frac{1}{n+4}}$}
\\ B_0^\alpha \big((n+4) B_0^2 t\big)^{-\beta} (1-x)& \mbox{for $1-x \ll  t^{-\frac{1}{n+4}}$}.
\end{array}\right.
$$
It is apparent from our numerical simulation (Fig.~\ch{%\ref{8-11S}\st{ and }\ref{9-10S}\st{, second column}
\ref{fig:correction-h1}, right}) that the correction should consist, also in this case, of a globally defined function $H_1$ and $s_1$ as in \ch{\eqref{exp-weak-1_bulk1.1_gen_n_bis}}, with the same time exponent $\beta$; however, at the moment this is left as an open question.

\medskip

As in the \ch{\st{critical} balanced} case \ch{\st{(cf. Remark)}}, a rigorous stability result is expected to hold \ch{for $\alpha\ne \frac{4}{n+3}$} and would be interesting to be pursued, showing convergence of solutions $H$ of \eqref{tfe_classical_fixed} to the corresponding $H_0$. In this respect, Fig. \ref{7Sb} \ch{and \ref{fig:dynamic_s}}
\begin{figure}[t]
    \centering

    \includegraphics[width=0.45\textwidth]{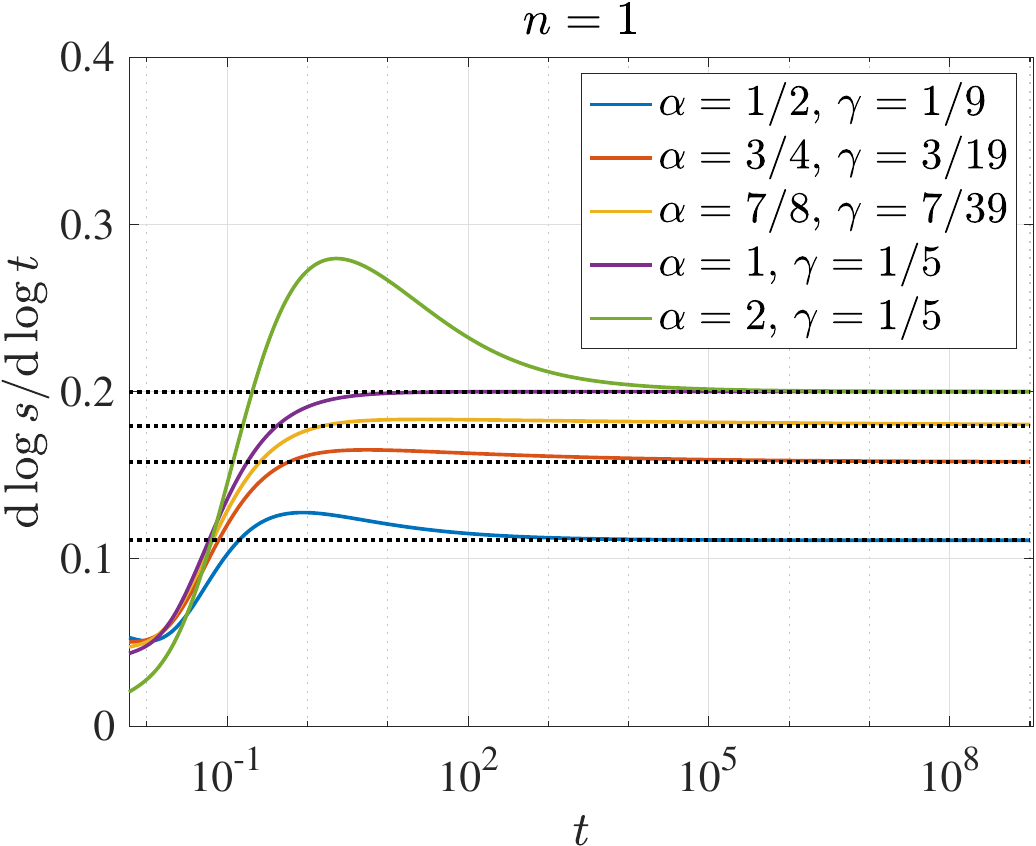}\hfill
    \includegraphics[width=0.45\textwidth]{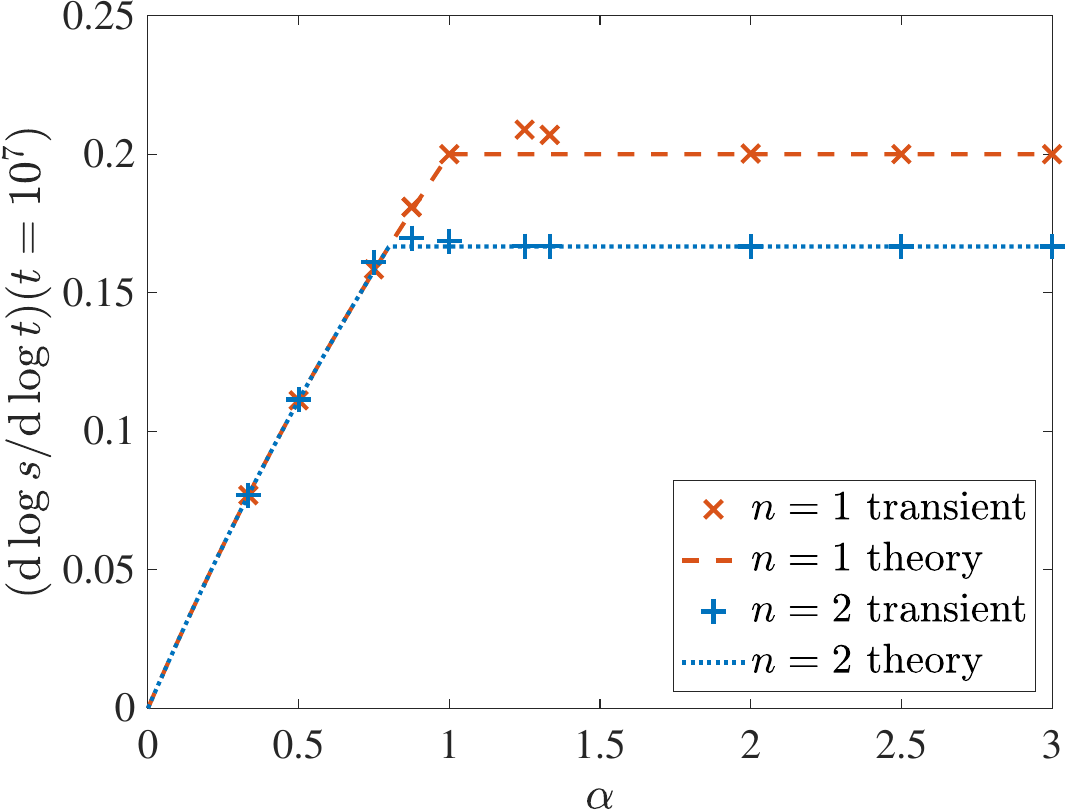}

    \caption{\ch{Convergence to self-similar solutions observed in the long-time behavior $s(t)\sim t^\gamma$ of numerical solutions to \eqref{tfe_classical_fixed} for mobility exponents $n=1,2$, where (left) we show the logarithmic time-derivative $\tfrac{\mathrm{d}\log s}{\mathrm{d}\log t}$ approaching the predicted $\gamma$ (black dotted line) for $n=1$ and (right) we show the evaluation of the logarithmic time-derivative at $t=10^7$ for various $\alpha$ and for $n=1,2$ compared to the theoretical prediction.}}
    \label{fig:dynamic_s}
\end{figure}
\ch{provide} numerical simulations supporting convergence, and \ch{Fig.~\ref{fig:correction-h1} and \ref{8-11S}}
\begin{figure}[t]
    \centering
    \includegraphics[width=0.315\textwidth]{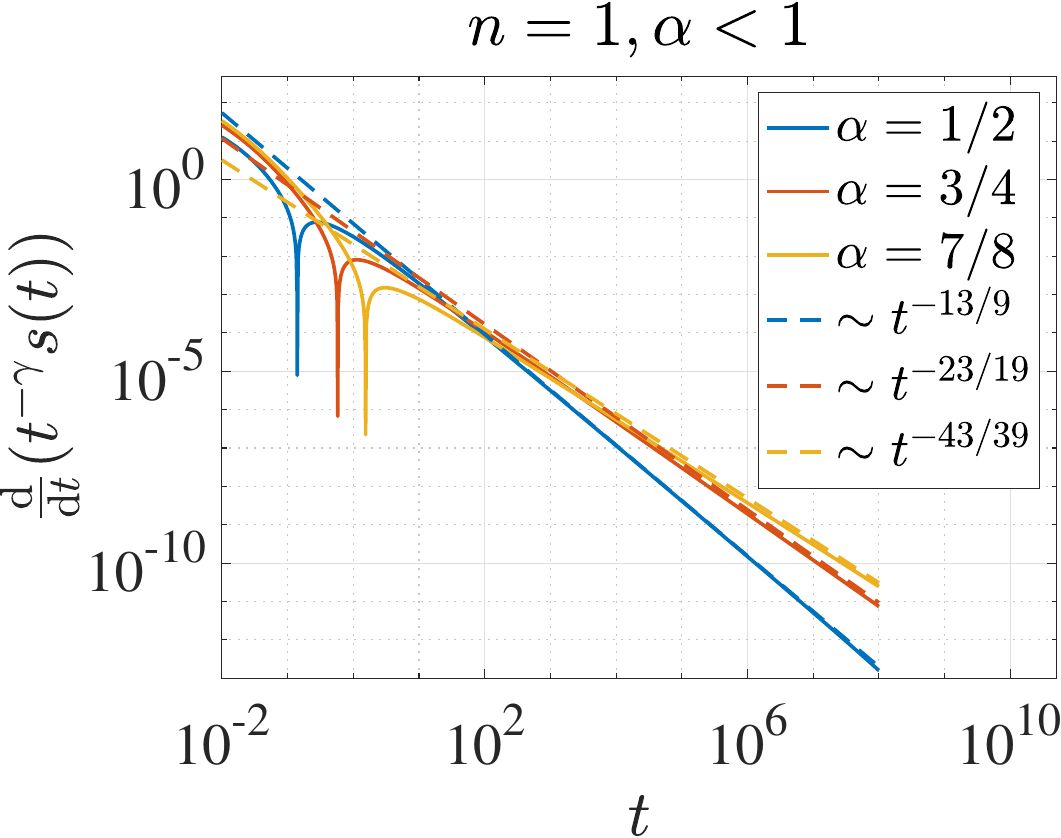}\hfill
    \includegraphics[width=0.315\textwidth]{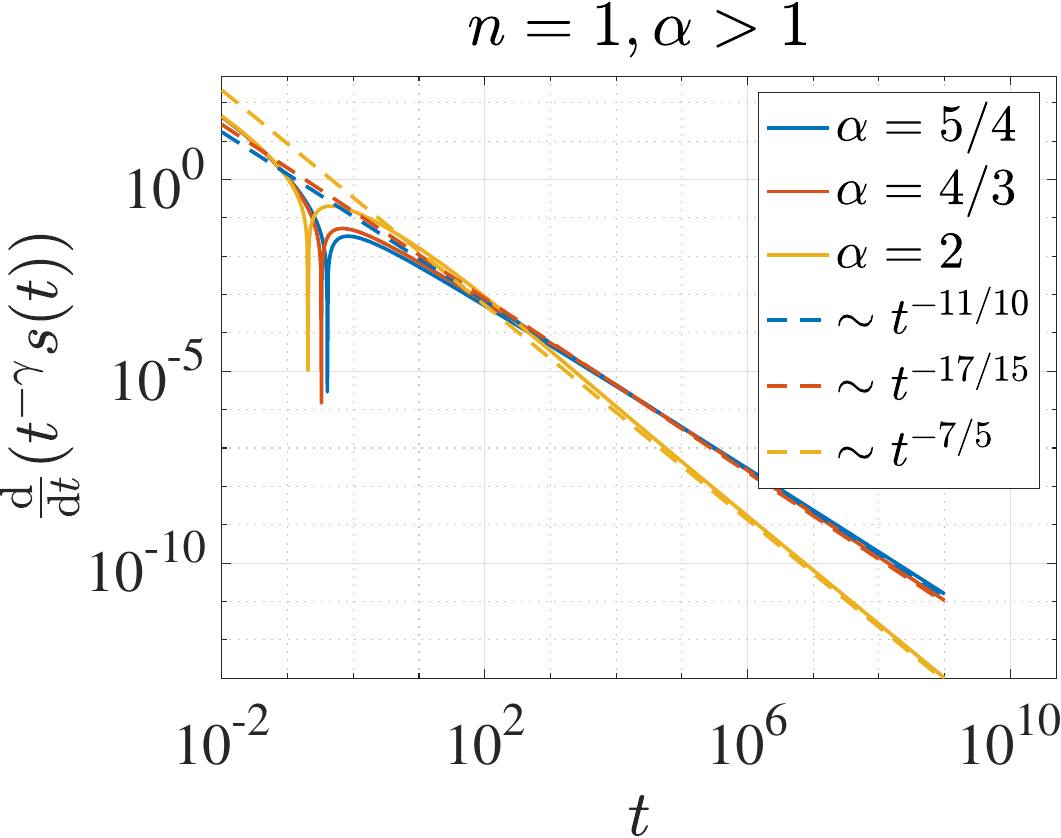}\hfill
    \includegraphics[width=0.315\textwidth]{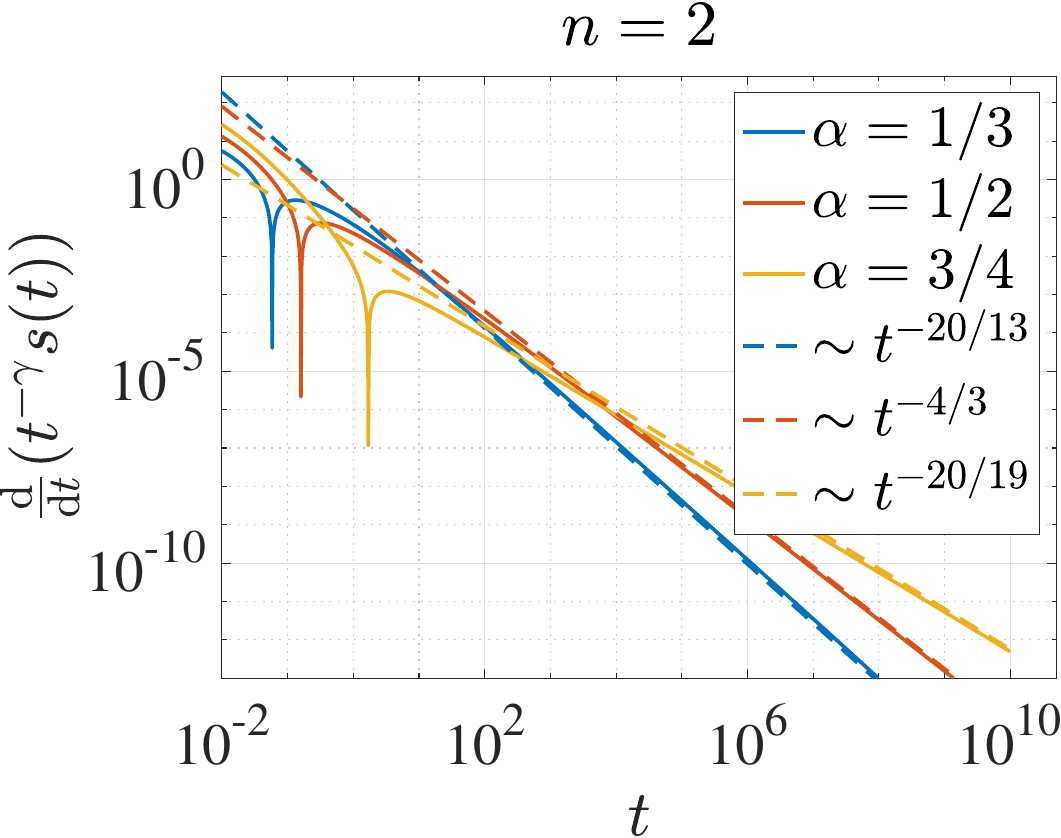}
    \caption{Time correction, in the form $\tfrac{\d}{\d t}(t^{-\gamma}s(t))$ versus $t$, \ch{for numerical solutions to \eqref{tfe_classical_fixed}}, shown using full lines for strong contact line friction $n=1$, $\alpha<1$ (left), weak contact line friction $n=1$, $\alpha>1$ \ch{(center),} and strong contact line friction $n=2$, $\alpha<4/5$ (right)\ch{. Dashed lines are the corresponding theoretical predictions.}
    }\label{8-11S}
\end{figure}
numerically validate the next-to-leading-order corrections to the macroscopic profile $H_0$.

\medskip

A final remark concerns dependence on the contact-line frictional coefficient $d$ (cf. \eqref{BC2-bis}), which we could scale out. Let $s_*$ be the asymptotic position of the free boundary of the solution to  \eqref{tfe_classical} with $D=1$ and $M=2$, as given by \ch{\eqref{s1.gen_n.s.f.bis}}
%
%\eqref{concl-lead-strong}
%
and \eqref{concl-lead-weak}. Keeping mass equal to two, but for a generic $d>0$, the asymptotic position of the contact line in the original equation \eqref{tfe_classical_unscaled} is given by
    $$
 s(t)= Y_* s_*(T_*^{-1} t), \qquad\mbox{where}\quad H_* \stackrel{(\ref{scaling2})}=Y_*^{-1}\stackrel{(\ref{scaling1})}= T_*^{-\frac{1}{n+4}}\stackrel{(\ref{D=1})}= d^{\frac{1}{4-\alpha(n+3)}}, \quad\alpha\ne \tfrac{4}{n+3},
    $$
    yielding
    $$
     s(t) \sim \left\{\begin{array}{rl}
    3^{\frac{1-\gamma}2}\gamma^{-\gamma} d^{-\frac{1}{\alpha+4}} t^\gamma & \mbox{if $\alpha<\frac{4}{n+3}$}
    \\
     \big((n+4) B_0^2 t\big)^{\frac{1}{n+4}} & \mbox{if $\alpha>\frac{4}{n+3}$}
\end{array}\right.        \qquad \mbox{as} \ t\to +\infty.
    $$
    This shows that in the strong case, as expected, larger frictional coefficients $d$ yield slower speeds; notably, however, in the weak case the leading-order asymptotic speed in \eqref{tfe_classical_unscaled} is instead oblivious of both $\alpha$ and $d$, hence universal.

%%%%%%%%%%%
\appendix
%%%%%%%%%%%

%
\section{The gradient-flow formulation and its discretization}\label{sec:numerics}
\renewcommand{\thesection}{\Alph{section}}\setcounter{section}{1}\setcounter{equation}{0}
%\section{The gradient-flow formulation and its discretization\label{sec:numerics}}
%
\newcommand{\oomega}{{\{h>0\}}}

\subsection{Gradient-flow formulation}

Problem \eqref{tfe_classical_unscaled} and its discretization are based on a gradient flow formulation for the height $h$ and the wetted area $\oomega=(s_-,s_+)$ as in \eqref{support}. The gradient flow
\begin{align}
    \label{eqn:evolution}
    \dot{h}=-\partial_\eta \Psi^*(h,{\rm D}\mathcal{E}[h]).
\end{align}
is {formally} defined in terms of the energy  $\mathcal{E}$ in \eqref{def-E} and {the} dual dissipation potential
\begin{align}
    \label{eqn:dissipation}
    \Psi^*(h,\eta)  =\frac12\int_\oomega  m(h)\, (\partial_y\pi)^2 {\rm d}y + \frac12\int_{\partial\oomega} \mcl(\partial_y h)\,\zeta^2 {\rm d}s,\qquad \eta=(\pi,\zeta),
\end{align}
where $\Psi^*(h,0)\equiv 0$ and $\Psi^*$ is convex in the second argument.
{The bulk mobility $m$ and the contact-line mobility $\mcl$ are non-negative functions, which in the case of \eqref{tfe_classical_unscaled} are given by}
\begin{align}\label{spec-m}
    m(h)=h^n \qquad\mbox{and}\qquad \mcl(z)=\tfrac{2}{d^{\frac 1 \alpha}}\ch{|z|}^{\frac 2 \alpha}.
\end{align}
The first term in $\Psi^*$ encodes the standard dissipation of the viscous fluid with nontrivial slip boundary conditions, whereas the second term encodes the extra dissipation at the contact line $y=s_\pm$. This formulation relies on the formal {assumption}
that for fixed time there exists
a representation of the dual force $\eta=(\pi:\oomega\to\mathbb{R},\zeta:\partial\oomega\to\mathbb{R})$ such that
\begin{align}
    \label{eqn:representation}
    \langle \eta ,\dot{h} \rangle = \int_\oomega \pi\dot{h}{\,{\rm d}y} + \int_{\partial\oomega}\zeta\dot{h}{\,{\rm d}s} \quad \mbox{for any rate $\dot{h}(t):\oomega\to\mathbb{R}$.}
\end{align}
Using this representation\footnote{For dual forces $\eta$ that do not admit such a representation, we formally set $\Psi^*(h,\eta)=\infty$.}, one identifies $\eta$ with $\langle \eta,\bar{v}\rangle=\langle{\rm D}\mathcal{E}[h],\bar{v}\rangle$ in a weak formulation, i.e.\ch{,}
\begin{subequations}
    \label{eqn:weakformulation}
    \begin{align}
        \label{eqn:weak1}
        \int_\oomega \pi \bar{v} {\,\rm d}y+ \int_{\partial\oomega} \zeta \bar{v} {\,\rm d}s= \int_\oomega \partial_y h \  \partial_y \bar{v} {\ \rm d}y+ \int_{\partial\oomega}\frac{1}{2|\partial_y h|}\big((\partial_yh)^2+(-2S)\big) \bar{v}{\,\rm d}s.
    \end{align}
  Testing the gradient flow \eqref{eqn:evolution} with $\bar{\eta}=(\bar{\pi},\bar{\zeta})$  and using \eqref{eqn:dissipation} gives $\langle\bar{\eta},\dot{h}\rangle=\langle\bar{\eta},-\partial_\eta\Psi^*(h,{\rm D}\mathcal{E}[h])\rangle$, or in full detail
    \begin{align}
        \label{eqn:weak2}
        \int_\oomega \dot{h}{\bar{\pi}} {\,\rm d}y+ \int_{\partial\oomega}\dot{h}{\bar{\zeta}} {\,\rm d}s= -\int_\oomega m(h) \ \partial_y\pi \ \partial_y\bar{\pi}\ {\rm d}y - \int_{\partial\oomega}\mcl (\partial_y h) \zeta \bar{\zeta} {\,\rm d}s.
    \end{align}
\end{subequations}
We seek $(\dot{h},\pi,\zeta)$ that satisfy \eqref{eqn:weakformulation} for all test functions $(\bar{v},\bar{\pi},\bar{\zeta})$. By testing this weak formulation with
$(\bar{v},\bar{\pi},\bar{\zeta})=(\dot{h},\pi,\zeta)$, we can deduce the energy descent
\begin{align}
    \tfrac{{\rm d}}{{\rm d}t}\mathcal{E}[h(t)]=\langle \eta,\dot{h}\rangle=-\Big(\int_\oomega m(h)\, (\partial_y\pi)^2{\,\rm d}y + \int_{\partial\oomega}\mcl (\partial_y h) \, \zeta^2{\,\rm d}s\Big)\le 0.
\end{align}
Additionally, the solution $h(t,y)$ and the contact line $s_\pm(t)$ satisfy the kinematic condition
\begin{align}
   \tfrac{\d}{\d t} h(t,s_\pm (t))=\dot{h}\big(t,s_\pm(t)\big)+\dot{s}_\pm(t) \partial_yh\big(t,s_\pm(t)\big)=0,
\end{align}
which we can use to reconstruct the boundary velocity and evolve the domain {$\oomega$} and the solution $h$.
Using integration by parts and assuming the solution is smooth enough, from \eqref{eqn:weak1} we identify  $\pi=-\partial_y^2 h$ and $\zeta=\tfrac{1}{2|\partial_y h|}((-2S)-(\partial_y h)^2)$. Using \eqref{eqn:weak2} we can recover the general evolution of $h$ and $\oomega$. We will now focus on the complete wetting case, $S=0$. Then the thin-film dynamics are governed by
\begin{subequations}
    \begin{align}
        \label{eqn:strong1}
         & \dot{h}=\partial_y(m(h)\,\partial_y\pi), \qquad                         &  & \pi=-\partial^2_{y}h           &  & \text{in }\oomega,         \\
         \label{eqn:strong2}
         & \dot{h}=-\dot{s}_\pm \partial_yh = -\mcl(\partial_y h)\,\zeta, \qquad &  & \zeta=-\tfrac12 |\partial_y h| &  & \text{on }\partial\oomega.
    \end{align}
\end{subequations}
with natural boundary conditions for $\pi$ in \eqref{eqn:strong1}.
Assuming \eqref{spec-m}, at the contact line $\partial\{h>0\}$ we get \ch{for positive speeds (which are of interest to us)}
\begin{equation}\label{clnew}
\dot{h}= 2 d^{-\frac 1 \alpha} |\partial_y h|^{\frac 2 \alpha}\big(\tfrac12|h_y|\big), \quad\mbox{or equivalently}\quad  \dot{s}_\pm= \pm d^{-\frac 1 \alpha}|\partial_y h|^{\frac 2 \alpha} \quad \mbox{at $\partial\oomega$},
\end{equation}
which coincides with \eqref{BC2-bis}. The Ren-E model with quadratic dissipation has $\alpha=1$\ch{;\st{and}} for a derivation from the Stokes problem see \cite{peschka2018variational}.

\subsection{Discretization}

The weak formulation \eqref{eqn:weakformulation} is discretized using a standard finite element discretization in space. In the derivative of the energy, we replace $\partial_y h$ by $\partial_y h + \tau\partial_y \dot{h}$ in order to achieve a semi-implicit treatment of the highest-order derivative in the time-discretization, a standard method in higher-order parabolic equations. For given $h$, we seek $(\dot{h},\pi,\zeta)$ using $P_1$ finite elements defined in $\oomega$ and on $\partial\oomega$, respectively, and solve
\begin{subequations}
    \label{eqn:discreteweak}
    \begin{align}
         & \int_\oomega\!\!\!\!\!\!\!\!\!\pi v - \tau \ \partial_y\dot{h}\ \partial_y\bar{v} {\,\rm d}y+ \int_{\partial\oomega}\!\!\!\!\!\!\!\!\!\!\zeta \bar{v} {\,\rm d}s  = \int_\oomega\!\!\!\!\!\!\!\!\! \partial_yh \ \partial_y\bar{v} {\,\rm d}y+ \int_{\partial\oomega}
         \!\!\tfrac12|\partial_y h|
         \bar{v}{\,\rm d}s,   \\
    & \int_\oomega\!\!\!\!\!\!\!\!\!\dot{h}{\bar{\pi}} {\,\rm d}y+ \int_{\partial\oomega}\!\!\!\!\!\!\!\!\!\dot{h}{\bar{\zeta}} {\,\rm d}s                 = -\int_\oomega\!\!\!\!\!\!\!\!\!m(h) \ \partial_y\pi \ \partial_y\bar{\pi}\ {\rm d}y - \int_{\partial\oomega}\!\!\!\!\!\!\!\!\!\mcl(\partial_y h)\, \zeta \bar{\zeta} {\,\rm d}s.
    \end{align}
\end{subequations}
where in the numerical scheme $m=m(h)$ and $\mcl=\mcl(\partial_y h)$ are evaluated explicitly from the previous time step.
Now, we introduce an arbitrary Lagrangian-Eulerian method by constructing a mapping $\xi:\oomega(t_0)\to\oomega(t)$ using the linear construction
\begin{align}
    \label{eqn:mapping}
    \xi(t,y) = \frac{s_+(t)-s_-(t)}{s_+(t_0)-s_-(t_0)} \bigl(y-s_-(t_0)\bigr) + s_-(t),
\end{align}
which allows to define the function $H(t):\oomega(t_0)\to\mathbb{R}$ via $H(t,y)=h(t,\xi(t,y))$. By construction we have $H(t,s_\pm(t_0))\equiv 0$ and correspondingly $\partial_t H =0$ at the fixed contact line $y=s_\pm(t_0)$. For the time derivatives we have $\dot{H}(t,y)=\dot{h}(t,\xi) + \partial_y h(t,\xi)\dot{\xi}$. The knowledge of $\dot{h}$ from the solution of \eqref{eqn:discreteweak} and the boundary condition $\dot{H}=0$ entirely determine the time derivatives $\dot{H}$ and the mapping $\dot{\xi}$.
Note that, for a moving front, $\partial_y h=0$ at $\partial\oomega$ is not an issue for the application of $P_1$ finite elements, since on the last element connected to $s_\pm$, \eqref{clnew} yields a possibly small but nonzero value of $|\partial_y h|$. %}
In the reference domain we can update both the height function $H(t+\tau,y)=H(t,y)+\tau \dot{H}(t,y)$ and the map $\xi(t+\tau,y)=\xi(t,y)+\tau \dot{\xi}(t,y)$, which uniquely determines $h(t,y)$ on a moving domain. A similar approach for partial wetting is studied in \cite{peschka2018variational} and higher-dimensional extensions are discussed in \cite{peschka2022model}.

%%%%%%%%%%%%%%%%%%%%%%%%%%%%%%%%%%%%%%%%%%%%%%
%
%\enlargethispage{20pt}
%
%%%%%%%%%% Insert bibliography here %%%%%%%%%%%%%%
%\vskip2pc
%\bibliographystyle{abbrv} %%%% .BST file
%\bibliography{existence-ref} %%%%% .Bib file

%
\bibliographystyle{unsrt}
\bibliography{friction_selfsimilar_arxiv}

\end{document}